\documentclass[12pt,reqno]{amsart}
\usepackage{indentfirst, amssymb, amsmath, amsthm, mathrsfs, setspace, indentfirst, enumerate,  mathrsfs, amsmath, amsthm}
\usepackage[bookmarksnumbered, colorlinks, plainpages]{hyperref}
\usepackage{mathrsfs}
\usepackage{cite}
\usepackage{tikz}
\usepackage{graphicx}
\usetikzlibrary{arrows.meta,calc,decorations.pathreplacing,positioning}
\usetikzlibrary{shapes.geometric}
\usepackage{float}
\usepackage{pgfplots}
\usetikzlibrary{shadows}
\pgfplotsset{compat=1.18}

\usepgfplotslibrary{fillbetween}
\definecolor{curveblue}{RGB}{30,90,220}
\definecolor{negred}{RGB}{255,182,193}
\definecolor{posgreen}{RGB}{40,150,70}
\definecolor{rootdark}{RGB}{120,40,20}
\usepgfplotslibrary{fillbetween}
\definecolor{c1}{RGB}{70,40,200}

\definecolor{c2}{RGB}{213,14,0}

\definecolor{c3}{RGB}{0,158,10}

\definecolor{c4}{RGB}{230,126,34}
\newtheorem*{theoA}{Theorem A}
\newtheorem*{theoB}{Theorem B}
\newtheorem*{theoC}{Theorem C}
\newtheorem*{theoD}{Theorem D}
\newtheorem*{theoE}{Theorem E}
\newtheorem*{theoF}{Theorem F}

\newtheorem*{cor A}{Corollary A}
\newtheorem*{cor B}{Corollary B}

\newtheorem{theo}{Theorem}[section]
\newtheorem{lem}{Lemma}[section]
\newtheorem{cor}{Corollary}[section]

\newtheorem{prob}{Problem}[section]
\newtheorem{defi}{Definition}[section]

\newcommand{\ol}{\overline}
\newcommand{\be}{\begin{equation}}
	\newcommand{\ee}{\end{equation}}
\newcommand{\beas}{\begin{eqnarray*}}
	\newcommand{\eeas}{\end{eqnarray*}}
\newcommand{\bea}{\begin{eqnarray}}
	\newcommand{\eea}{\end{eqnarray}}

\numberwithin{equation}{section}

\begin{document}
\title[C\MakeLowercase {oefficient estimates and}......]{\LARGE C\Large\MakeLowercase {oefficient estimates and}  B\MakeLowercase{ohr Phenomenon for Pluriharmonic Mappings in the polydisc}}

\date{}
\author[S. M\MakeLowercase{ajumder}, A. B\MakeLowercase{anerjee}, S. P\MakeLowercase{anja} \MakeLowercase{and} D. P\MakeLowercase{ramanik} ]{S\MakeLowercase{ujoy} M\MakeLowercase{ajumder}$^{1}$, A\MakeLowercase{bhijit} B\MakeLowercase{anerjee}$^2$, S\MakeLowercase{hantanu} P\MakeLowercase{anja}$^3$$^*$ \MakeLowercase{and} D\MakeLowercase{ebabrata} P\MakeLowercase{ramanik}$^4$}

\address{$^{1}$ Sujoy Majumder, Department of Mathematics, Raiganj University, Raiganj, West Bengal-733134, India.}
\email{sm05math@gmail.com, sjm@raiganjuniversity.ac.in}
\address{$^{2}$ Abhijit Banerjee, Department of Mathematics, University of Kalyani, West Bengal 741235, India.}
\email{abanerjee\_kal@yahoo.co.in, abhijitbanerjee@klyuniv.ac.in}
\address{$^3$ Shantanu Panja, Department of Mathematics, University of Kalyani, West Bengal 741235, India.}
\email{panjasantu07@gmail.com}
\address{$^{4}$ Debabrata Pramanik, Department of Mathematics, Raiganj University, Raiganj, West Bengal-733134, India.}
\email{debumath07@gmail.com}

\renewcommand{\thefootnote}{}
\footnote{*\emph{Corresponding Author}: Shantanu Panja.}

\renewcommand{\thefootnote}{\arabic{footnote}}
\setcounter{footnote}{0}

\begin{abstract}
	We introduce the class $\mathscr{P}_{\mathcal{H}_n^0}(\alpha)$ $(0\leq \alpha<1)$ of normalized pluriharmonic mappings in the setting of several complex variables. This class extends the harmonic family $\mathscr{P}_{\mathcal{H}}^{0}(\alpha)$ to the multidimensional framework. We establish sharp coefficient estimates and growth theorems for functions in $\mathscr{P}_{\mathcal{H}_n^0}(\alpha)$, thereby generalizing the corresponding results of Li and Ponnusamy \cite{Li-Ponnusamy-2013a} and Allu and Halder \cite{Allu-Halder-2021}. We further determine the associated Bohr radius and investigate the sections (partial sums) of functions in this class, obtaining quantitative results that describe the behavior of their truncated expansions.
\end{abstract}
\thanks{Typeset by \AmS -\LaTeX}
\maketitle
\section{\bf Introduction}
Pluriharmonic mappings form one of the fundamental classes of mappings in several complex variables, providing a natural higher-dimensional analogue of planar harmonic mappings. Every pluriharmonic mapping admits a local decomposition into the sum of a holomorphic and an anti-holomorphic mapping, thereby inheriting many of the geometric characteristics of holomorphic functions while allowing considerably greater flexibility. Owing to this rich structure, pluriharmonic mappings have become an indispensable tool in geometric function theory, where they have been extensively employed in the study of univalence criteria, coefficient estimates, growth and distortion theorems, covering results, Landau--Bloch type theorems, Lipschitz continuity and various extremal problems. Their intimate connection with complex geometry and elliptic partial differential equations continues to stimulate substantial research activity in higher-dimensional complex analysis.\vspace{1.2mm}

Beyond their intrinsic mathematical significance, pluriharmonic mappings possess important connections with mathematical physics. Harmonic and pluriharmonic functions arise naturally as solutions of elliptic boundary value problems describing equilibrium phenomena, including electrostatic and gravitational potentials, steady-state heat conduction and incompressible fluid flow. Moreover, pluriharmonic maps play a significant role in complex differential geometry, particularly in the study of K\"ahler manifolds, energy-minimizing variational problems, nonlinear sigma models and minimal surfaces. Such geometric structures also appear in several mathematical formulations of general relativity and string theory, where complex analytic techniques provide an effective framework for investigating field configurations and geometric properties of the underlying spaces. These diverse applications further underscore the importance of understanding the geometric behaviour of pluriharmonic mappings.\vspace{1.2mm}

Motivated by these developments, considerable attention has recently been devoted to coefficient problems for pluriharmonic mappings. Among the various coefficient phenomena, the celebrated theorem of Bohr occupies a central position. In his pioneering work of 1914, Bohr discovered the remarkable fact that if $f(z)=\sum_{n=0}^{\infty}a_nz^n$ is analytic and satisfies $|f(z)|<1$ in the unit disk $\mathbb{D}$, then
\[
\sum_{n=0}^{\infty}|a_n|r^n\le1,
\]
whenever $r\le1/3$, where the constant $1/3$ is best possible. This unexpected relationship between the modulus of a bounded analytic function and the majorant of its Taylor coefficients has become one of the cornerstones of geometric function theory.\vspace{1.2mm}

During the past century, Bohr's inequality has undergone remarkable developments and now extends far beyond its original formulation. Sharp Bohr radii have been established for numerous families of analytic, harmonic and pluriharmonic mappings, as well as for subordinate functions, weighted coefficient problems, operator-valued mappings and holomorphic functions on multidimensional domains. These advances have revealed profound interactions between the Bohr phenomenon, geometric function theory, functional analysis and several complex variables, thereby transforming Bohr's original theorem into a fundamental principle of modern analysis.\vspace{1.2mm}

Throughout this paper we work in the setting of several complex variables. For the convenience of the reader, we first summarize the notation and basic concepts that will be used repeatedly in the sequel.\par

For $a=(a_1,\ldots,a_n)\in\mathbb C^n$ and
$r=(r_1,\ldots,r_n)\in\mathbb R^n$ with $r_j>0$, define the open polydisc by
$\mathbb P\Delta(a;r)=\left\{z\in\mathbb C^n:|z_j-a_j|<r_j,\;
j=1,\ldots,n\right\}.$ The vector $a$ is referred to as the centre, whereas $r$ is called the corresponding polyradius.  In particular, $\mathbb{P}\Delta(0;1)=\mathbb{P}\Delta(0_n;1_n)$, where $0_n=(0,0,\ldots,0)$ and $1_n=(1,1,\ldots,1)$. This polydisc is called the unit polydisc in $\mathbb{C}^n$. The unit disk in the complex plane is denoted by $\mathbb{D}$. 

Let $\Omega_n\subset\mathbb C^n$ be a simply connected domain containing
$\mathbb P\Delta(0;1)$.
We denote by
$\mathcal F_n(\Omega_n)$
the family of holomorphic functions on $\Omega_n$.
Furthermore, $$\mathcal B_n(\Omega_n)=\left\{f\in\mathcal F_n(\Omega_n):f(\Omega_n)\subseteq\overline{\mathbb D}
\right\}$$ denotes the class of bounded holomorphic mappings from $\Omega_n$ into the closed unit disc.
When $n=1$ we simply write $\mathcal H(\Omega)=\mathcal H_1(\Omega_1)$ for the classical family of analytic functions on $\Omega$.

The standard multi-index notation will be employed throughout.
\par For $\alpha=(\alpha_1,\ldots,\alpha_n)\in\mathbb N_0^n$,
its order and factorial are respectively defined by
$|\alpha|=\sum_{j=1}^{n}\alpha_j,\qquad\alpha!=\alpha_1!\cdots\alpha_n!.$

Correspondingly, for $z=(z_1,\ldots,z_n)\in\mathbb C^n$,
$z^\alpha=\prod_{j=1}^{n}z_j^{\alpha_j},\qquad|z|^\alpha
=\prod_{j=1}^{n}|z_j|^{\alpha_j}.$

\smallskip
Every function
$f\in\mathcal F_n(\mathbb P\Delta(0;1))$
admits an absolutely convergent homogeneous power-series expansion in $z_1,\ldots,z_n$, of the form

\begin{align}\label{In-1.1}
	f(z)=\sum\limits_{\alpha_1,\alpha_2,\ldots,\alpha_n=0}^{\infty} a_{\alpha_1,\alpha_2,\ldots,\alpha_n}z_1^{\alpha_1}z_2^{\alpha_2}\ldots z_n^{\alpha_n}=
	\sum\limits_{m=0}^{\infty}\sum\limits_{|\alpha|=m} a_{\alpha}z^{\alpha}=\sum\limits_{|\alpha|=0}^{\infty} P_{|\alpha|}(z),
\end{align}
which is absolutely convergent in $\mathbb{P}\Delta(0;1)$, where the term $P_k(z)$ is a homogeneous polynomial in $z_1,z_2,\ldots,z_n$ of degree $k$.

\subsection{\bf {Basic ideas of pluriharmonic mapping}}
To introduce pluriharmonic mappings, we express each complex coordinate as
$z_j=x_j+iy_j,\qquad j=1,\ldots,n,$
where $x_j,y_j\in\mathbb R$.
For a complex-valued mapping
$f(z)=u(x,y)+iv(x,y),$ the functions $u$ and $v$ denote its real and imaginary parts, respectively, with $x=(x_1,\ldots,x_n)$ and $y=(y_1,\ldots,y_n)$.

The Cauchy--Riemann equations for each $z_j\;(j=1,\dots,n)$ are
\begin{align}\label{Eq 1.1}
	\frac{\partial u}{\partial x_j}
	=
	\frac{\partial v}{\partial y_j}\quad \text{and}\quad
	\frac{\partial u}{\partial y_j}
	=
	-\,\frac{\partial v}{\partial x_j}
	\qquad (j=1,\dots,n).
\end{align}

By differentiating (\ref{Eq 1.1}) with respect to $x_k$ and $y_k$, we see that both $u$ and $v$ satisfy the
following system of partial differential equations of second order:
\begin{align}\label{Eq 1.2}
	\frac{\partial^2}{\partial x_j \partial x_k}
	+
	\frac{\partial^2}{\partial y_j \partial y_k}
	= 0
	\quad \text{and} \quad
	\frac{\partial^2}{\partial x_j \partial y_k}
	-
	\frac{\partial^2}{\partial x_k \partial y_j}
	= 0
	\qquad (j,k = 1,\dots,n).
\end{align}

For a complex variable $z_j = x_j + i y_j$, we define
\begin{align}\label{Eq 1.3}
	\frac{\partial}{\partial z_j}
	= \frac{1}{2}\left( \frac{\partial}{\partial x_j}
	- i \frac{\partial}{\partial y_j} \right),\quad
	\frac{\partial}{\partial \bar z_j}
	= \frac{1}{2}\left( \frac{\partial}{\partial x_j}
	+ i \frac{\partial}{\partial y_j} \right).
\end{align}
\begin{align*}
	\partial =\sum\limits_{j=1}^n \frac{\partial }{\partial z_j} d z_j,\quad \ol{\partial} =\sum\limits_{j=1}^n \frac{\partial }{\partial \ol z_j} d \ol {z}_j \quad \text{and} \quad d=\partial+\ol{\partial}.
\end{align*}

Combining the operators introduced in \eqref{Eq 1.3} with straightforward differentiation yields 
\begin{align}\label{Eq 1.5}
	4\frac{\partial^2 f(z)}{\partial \bar z_j \partial z_k}=&\frac{\partial^2 u(x,y)}{\partial x_j x_k}+\frac{\partial^2 u(x,y)}{\partial y_j y_k}+i\left(\frac{\partial^2 v(x,y)}{\partial x_j x_k}+\frac{\partial^2 v(x,y)}{\partial y_j y_k}\right)\\&-i\left(\frac{\partial^2 u(x,y)}{\partial x_j y_k}-\frac{\partial^2 u(x,y)}{\partial x_k y_j}\right)+\left(\frac{\partial^2 v(x,y)}{\partial x_j y_k}-\frac{\partial^2 v(x,y)}{\partial x_k y_j}\right).\nonumber
\end{align}

Finally, let $\Omega_n$ be an open subset of $\mathbb{C}^n$. A mapping $f:\Omega_n\rightarrow\mathbb{C}$ is holomorphic if and only if it satisfies the
Cauchy--Riemann equations in complex form, namely,
\begin{align}\label{Eq 1.4}
	\bar{\partial}f=0,
	\qquad\text{or equivalently}\qquad
	\frac{\partial f(z)}{\partial\bar z_j}=0,
	\quad j=1,2,\ldots,n,
\end{align} throughout $\Omega_n$.

\subsection*{{\bf Pluriharmonic mappings}}

Let $\phi(x,y)$ be a real-valued function,
$x=(x_1,\ldots,x_n)$ and $y=(y_1,\ldots,y_n)$. We say that $\phi$ is \emph{pluriharmonic} if it satisfies the conditions \eqref{Eq 1.2}. Consequently, a continuous complex-valued function
$f(z)=u(x,y)+iv(x,y),$
defined on a domain $\Omega_n\subset\mathbb{C}^n$, is called a \emph{pluriharmonic mapping} whenever both $u(x,y)$ and $v(x,y)$ are real-valued pluriharmonic functions in $\Omega_n$.
Furthermore, if the functions $u(x,y)$ and $v(x,y)$ satisfy \eqref{Eq 1.1}, then $v(x,y)$ is referred to as a \emph{pluriharmonic conjugate} of $u(x,y)$.

\par
Now assume that $f(z)=u(x,y)+iv(x,y),$
where $u$ and $v$ possess continuous second-order partial derivatives.
It follows immediately from \eqref{Eq 1.4} and \eqref{Eq 1.5} that, whenever
$f$ is pluriharmonic, each complex partial derivative
$\frac{\partial f(z)}{\partial z_j},\; j=1,2,\ldots,n,$ is holomorphic in $\Omega_n$.

\vspace{1.5mm}

\par
Suppose next that $\Omega_n\subset\mathbb{C}^n$ is simply connected and let
$f$ be a complex-valued pluriharmonic mapping in $\Omega_n$. Since
$\partial f(z)/\partial z_j$ is holomorphic for every
$j=1,2,\ldots,n$, there exists a holomorphic function $h(z)$ on
$\Omega_n$ satisfying $\frac{\partial h(z)}{\partial z_j}=\frac{\partial f(z)}{\partial z_j},\;
j=1,2,\ldots,n.$
Define
$g(z)=\overline{f(z)}-\overline{h(z)}.$
Then, by the choice of $h$,
$$
\frac{\partial g(z)}{\partial\overline{z_j}}=
\overline{\frac{\partial f(z)}{\partial z_j}}-
\overline{\frac{\partial h(z)}{\partial z_j}}
=0, \qquad z\in\Omega_n,\quad j=1,2,\ldots,n.$$
Hence $g$ is holomorphic in $\Omega_n$. Consequently, every pluriharmonic mapping admits the canonical decomposition
$$f(z)=h(z)+\overline{g(z)},$$
where both $h$ and $g$ are holomorphic functions in $\Omega_n$.
\subsection{{\bf Different classes of pluriharmonic mappings}}

Let $\mathcal{H}_n$ denote the family of complex-valued pluriharmonic mappings
defined on the unit polydisk $\mathbb{P}\Delta(0;1)$ and normalized by
\[
f(0)=0
\quad \text{and} \quad
\nabla f(0):=
\left(
\frac{\partial f(0)}{\partial z_1},
\ldots,
\frac{\partial f(0)}{\partial z_n}
\right)
=(1,1,\ldots,1).
\]

When $\dim(\mathbb{C}^n)=1$, the class
$\mathcal{H}(=\mathcal{H}_1)$ coincides with the family of normalized
complex-valued harmonic mappings in the unit disk $\mathbb{D}$ satisfying
\[
f(0)=0
\quad \text{and} \quad
\frac{\partial f(0)}{\partial z}=1.
\]

As established in the preceding subsection, every mapping
$f\in\mathcal{H}_n$ admits the canonical decomposition
\[
f=h+\overline{g},
\]
where both $h$ and $g$ are holomorphic in
$\mathbb{P}\Delta(0;1)$. Throughout this paper, $h$ and $g$ will be referred
to as the \emph{holomorphic part} and the \emph{co-holomorphic part} of $f$,
respectively. Their corresponding power series expansions are given by
\[
h(z)
=
\sum_{j=1}^{n}z_j
+
\sum_{k=2}^{\infty}
\sum_{|\beta|=k}
a_{\beta}z^{\beta},
\qquad
g(z)
=
\sum_{k=1}^{\infty}
\sum_{|\beta|=k}
b_{\beta}z^{\beta}.
\]

Observe that, in the special case $g\equiv0$, the class
$\mathcal{H}_n$ reduces to the family $\mathcal{A}_n$ consisting of
holomorphic mappings in $\mathbb{P}\Delta(0;1)$ satisfying the normalization
conditions
\[
f(0)=0
\quad \text{and} \quad
\nabla f(0)
=
(1,1,\ldots,1).
\]
Moreover, if $\dim(\mathbb{C}^n)=1$, then
$\mathcal{A}(=\mathcal{A}_1)$ is precisely the classical family of normalized
holomorphic functions in the unit disk $\mathbb{D}$ with
\[
f(0)=0
\quad \text{and} \quad
\frac{\partial f(0)}{\partial z}=1.
\]

A particularly important subclass of $\mathcal{H}_n$ is defined by
\[
\mathcal{H}_n^{0}
=
\left\{
f\in\mathcal{H}_n:
\overline{\partial}f(0)=0
\right\}.
\]

Accordingly, if
$f=h+\overline{g}\in\mathcal{H}_n^{0}$,
then the holomorphic and co-holomorphic components admit the expansions
\begin{align}\label{Eq 1.7}
	h(z)
	&=
	\sum_{j=1}^{n}z_j
	+
	\sum_{k=2}^{\infty}
	\sum_{|\beta|=k}
	a_{\beta}z^{\beta},
	\qquad
	g(z)
	=
	\sum_{k=2}^{\infty}
	\sum_{|\beta|=k}
	b_{\beta}z^{\beta},
\end{align}
respectively. The class $\mathcal{H}_n^{0}$ constitutes the principal setting
for the present investigation.
Let $\mathcal{S}_{\mathcal{H}}$ denote the class of univalent and
sense-preserving harmonic mappings belonging to $\mathcal{H}$.

Among the subclasses of $\mathcal{H}_n$, those determined by geometric
properties of the image domain play a fundamental role. A pluriharmonic
mapping $f\in\mathcal{H}_n$ is called \emph{starlike} if the image
$f(\mathbb{P}\Delta(0;1))$ is a starlike domain with respect to the
origin. The corresponding family of pluriharmonic starlike mappings is denoted by $\mathcal{S}^{*}_{\mathcal{H}_n}.$

Similarly, a mapping $f\in\mathcal{H}_n$ is said to be
\emph{convex} whenever the image domain
$f(\mathbb{P}\Delta(0;1))$ is convex. The collection of all such mappings will be denoted by
$\mathcal{K}_{\mathcal{H}_n}.$

Furthermore, a mapping $f\in\mathcal{H}_n$ is called
\emph{close-to-convex} if $f(\mathbb{P}\Delta(0;1))$ is a close-to-convex domain.
The corresponding class is denoted by
$\mathcal{C}_{\mathcal{H}_n}.$

We further write
$\mathcal{S}^{*0}_{\mathcal{H}_n}$,
$\mathcal{K}^{0}_{\mathcal{H}_n}$ and
$\mathcal{C}^{0}_{\mathcal{H}_n}$
for the subclasses of
$\mathcal{S}^{*}_{\mathcal{H}_n}$,
$\mathcal{K}_{\mathcal{H}_n}$ and
$\mathcal{C}_{\mathcal{H}_n}$,
respectively, satisfying the normalization
$\overline{\partial}f(0)=(0,0,\ldots,0).$

\vspace{1.2mm}

The systematic study of harmonic univalent mappings was initiated by
Clunie and Sheil-Small in their pioneering paper
\cite{Clunie-Sheil-Small-1984}, where the class
$\mathcal{S}_{\mathcal{H}}$ together with several of its important geometric subclasses was introduced and investigated. Since then,
considerable attention has been devoted to these families as well as to their higher-dimensional counterparts. In particular, the classes
$\mathcal{S}^{*}_{\mathcal{H}_n}$,
$\mathcal{K}_{\mathcal{H}_n}$,
$\mathcal{C}_{\mathcal{H}_n}$,
$\mathcal{S}^{*0}_{\mathcal{H}_n}$,
$\mathcal{K}^{0}_{\mathcal{H}_n}$ and
$\mathcal{C}^{0}_{\mathcal{H}_n}$
have been investigated extensively by many authors; see
\cite{Hernandez-Martin-2013}-\cite{Liu-Yang-2019},
\cite{Muhanna-CVEE-2010},
\cite{Ponnusamy-Allu-Vuorinen-2009}-\cite{Ponnusamy-Kaliraj-Starkov-2017}.
\vspace{1.2mm}
In 1962, MacGregor \cite{MacGregor-1962} introduced the following class of
normalized analytic functions:
\begin{align*}
	\mathscr{P}:=\left\{f\in\mathcal{A}:	\Re\bigl(f'(z)\bigr)>0,\;\;	z\in\mathbb{D}
	\right\}.
\end{align*}

Motivated by this analytic family, Ponnusamy \emph{et al.} \cite{Ponnusamy-Yamamoto-Yanagihara-CVEE-2013}
proposed the following harmonic analogue:
\begin{align*}
	\mathscr{P}_{\mathcal{H}}:=\left\{
	f=h+\overline{g}\in\mathcal{H}:
	\Re\bigl(h'(z)\bigr)>|g'(z)|,\;\;z\in\mathbb{D}
	\right\},
\end{align*}
together with its normalized subclass
\begin{align*}
	\mathscr{P}_{\mathcal{H}}^{0}:=\left\{
	f=h+\overline{g}\in\mathscr{P}_{\mathcal{H}}:
	g'(0)=0	\right\}.\end{align*}

It was shown in
\cite{Ponnusamy-Yamamoto-Yanagihara-CVEE-2013}
that every mapping belonging to
$\mathscr{P}_{\mathcal{H}}$
is close-to-convex in the unit disk $\mathbb{D}$.
Moreover, a harmonic mapping $f=h+\overline{g}$
belongs to $\mathscr{P}_{\mathcal{H}}^{0}$
if and only if the analytic function
$h+\lambda g$ belongs to $\mathscr{P}$
for every unimodular constant $\lambda$
($|\lambda|=1$); see
\cite{Li-Ponnusamy-2013,Li-Ponnusamy-2013a}.
Making essential use of this characterization,
Li and Ponnusamy \cite{Li-Ponnusamy-2013}
derived coefficient estimates together with the radius of convexity for functions in $\mathscr{P}_{\mathcal{H}}^{0}$.

Inspired by the above developments, Li and Ponnusamy
\cite{Li-Ponnusamy-2013a} introduced the parameterized family
\begin{align*}
	\mathscr{P}_{\mathcal{H}}(\alpha):=
	\left\{	f=h+\overline{g}\in\mathcal{H}:
	\Re\bigl(h'(z)-\alpha\bigr)
	>|g'(z)|,\;\;z\in\mathbb{D}	\right\},
\end{align*}
where $\alpha\in[0,1)$, together with its normalized subclass
\begin{align*}
	\mathscr{P}_{\mathcal{H}}^{0}(\alpha):=
	\left\{f=h+\overline{g}\in	\mathscr{P}_{\mathcal{H}}(\alpha):g'(0)=0\right\}.
\end{align*}

It is immediate from the definitions that
$\mathscr{P}_{\mathcal{H}}(\alpha)
\subseteq \mathscr{P}_{\mathcal{H}}
\quad\text{and}\quad \mathscr{P}_{\mathcal{H}}^{0}(\alpha)
\subseteq \mathscr{P}_{\mathcal{H}}^{0}.$

Li and Ponnusamy \cite{Li-Ponnusamy-2013a}
further established that every function in
$\mathscr{P}_{\mathcal{H}}^{0}(\alpha)$
is univalent whenever $0\leq\alpha<1$.
They also obtained coefficient estimates for functions belonging to this class. For the reader's convenience, we next recall the corresponding results from
\cite{Li-Ponnusamy-2013a}.
 
 \begin{theoA}\emph{\cite[Theorem 1]{Li-Ponnusamy-2013a}}
 Let $f=h+\overline{g}\in \mathscr{P}_{\mathcal{H}}^{0}(\alpha)$ be of the form 
 \begin{align}\label{eq.1} 
 	h(z)=z+\sum_{m=2}^{\infty} a_mz^m \ \text{and} \ g(z)=\sum_{m=2}^{\infty} b_mz^m.
 \end{align} 
 Then for each $m\ge 2$,
 \begin{align*}
 |b_m|\le \frac{1-\alpha}{m}.
 \end{align*}
 
 The result is sharp, with $f(z)=z+\frac{1-\alpha}{m}\,\overline{z^{m}}$
 being extremal.
 \end{theoA}
 
 \begin{theoB}\emph{\cite[Theorem 2]{Li-Ponnusamy-2013a}}
 Let $f=h+\overline{g}\in \mathscr{P}_{\mathcal{H}}^{0}(\alpha)$ and be of the form \emph{(\ref{eq.1})}. Then for any $m\ge 2$,
 	\par \emph{(i)} $|a_m|+|b_m|\le\frac{2(1-\alpha)}{m}$;
 	\par \emph{(ii)} $\bigl||a_m|-|b_m|\bigr|\le\frac{2(1-\alpha)}{m}$;
 	\par \emph{(iii)}	$|a_m|\le\frac{2(1-\alpha)}{m}$.
 	
 	All the results are sharp, with $(1-\alpha)\bigl(-z-2\log(1-z)\bigr)+\alpha z$
 	being the extremal.
 \end{theoB}
 
In $2021$, Allu and  Halder \cite{Allu-Halder-2021} 
 proved the following sharp growth estimate for the class $\mathscr{P}_{\mathcal{H}}^{0}(\alpha)$.
 \begin{theoC}\emph{\cite[Theorem 2.1]{Allu-Halder-2021}}
 	Let $f=h+\overline{g}\in\mathscr{P}_{\mathcal{H}}^{0}(\alpha)$ with $0\le\alpha<1$. Then
 	\begin{align*}
 	|z|
 	+\sum_{m=2}^{\infty}
 	\frac{2(1-\alpha)(-1)^{m-1}}{m}|z|^{m}
 	\le
 	|f(z)|
 	\le
 	|z|
 	+\sum_{m=2}^{\infty}
 	\frac{2(1-\alpha)}{m}|z|^{m}.
 	\end{align*}
 	Both the inequalities are sharp.
 \end{theoC}
 
\subsection{\bf{Bohr phenomenon for the complex-valued pluriharmonic mappings}}
For a given power series of the form \eqref{In-1.1}, its majorant series is defined by
\begin{align*}
	M_f(r)=\sum\limits_{m=0}^{\infty}\sum\limits_{|\alpha|=m} |a_{\alpha}z^{\alpha}|=\sum\limits_{m=0}^{\infty}\sum\limits_{|\alpha|=m} |a_{\alpha}|r^{\alpha},
\end{align*}
where $r=(|z_1|,|z_2|,\ldots,|z_n|)=(r_1,r_2,\ldots,r_n)$ such that $r_j<1$ for $j=1,2,\ldots,n$. Clearly, for each $f\in\mathcal{B}(\mathbb{D})$, we write $M(r):=M_f(r)$ which is an increasing function of $r$ for $0\leq r<1$. Moreover, $M_f(0)=|a_0|=|f(0)|\leq 1$. It is worth noting that, for certain functions $f\in\mathcal{B}(\mathbb{D})$, the majorant series $M_f(r)$ exceeds $1$ (see \cite{Boas-2000}). This naturally leads to the following question: for which values of $r\in[0,1)$ does the inequality
\begin{align}\label{In-1.3}
	M_f(r)=\sum\limits_{m=0}^{\infty}|a_m|r^m\leq 1
\end{align}
hold for every $f\in\mathcal{B}(\mathbb{D})$? The inequality \eqref{In-1.3} on $[0,1)$ is
usually known as Bohr inequality for the class $\mathcal{B}(\mathbb{D})$. 

In $1914$, Harald Bohr \cite{Bohr-1914} observed that the inequality \eqref{In-1.3} is true for $|z|\leq 1/6$ and this practice was further expanded by Wiener, Riesz, and Schur who
independently established the inequality \eqref{In-1.3} on the disk $|z|=r\leq 1/3$. Proofs
have also been given by Sidon \cite{Sidon-1927} and Tomi\'c \cite{Tomic-1962} and it is described
precisely as follows:

\begin{theoD}If $f\in \mathcal{B}(\mathbb{D})$, then $M_f(r)\leq 1$ for $0\leq r\leq 1/3$. The number $1/3$ is best possible.
\end{theoD}

The inequality $M_f(r)\leq 1$ for $f\in \mathcal{B}(\mathbb{D})$ does not hold for any $r>1/3$. This can be verified by considering the M\"{o}bius transformation
\begin{align*}
	\phi_a(z)=\frac{a-z}{1-az}
\end{align*}
where $a\in (0,1)$ is chosen sufficiently close to $1$. 

We note that the inequality (\ref{In-1.3}) equivalent to
\begin{align}\label{In-1.4}
	\sum\limits_{m=1}^{\infty}|a_m|r^m\leq 1-|a_0|=d\left(f(0),\partial f(\mathbb{D})\right)\quad \mbox{for}\quad |z|\leq \frac{1}{3},
\end{align}
where $d\left(f(0),\partial f(\mathbb{D})\right)$  denotes the Euclidean distance between $f(0)$ and the boundary of $f(\mathbb{D})$. Functions $f\in \mathcal{B}(\mathbb{D})$ satisfying (\ref{In-1.4}), sometimes are said to satisfy the classical Bohr’s phenomenon. It is important to note that the existence of the radius $1/3$ in (\ref{In-1.4}) is independent of the coefficients of the power series $f(z)=\sum_{m=0}^{\infty} a_mz^m$,  i.e., in a better way we can demonstrate this fact by saying that a Bohr phenomenon appears in the class of analytic self-maps of the unit disk $\mathbb{D}$.\vspace{1.2mm}

The theory of Bohr inequalities has become an important and active area of research in complex analysis, geometric function theory, and operator theory.  Since the discovery of Theorem D, Bohr's theorem has inspired extensive research leading to numerous generalizations, refinements, and applications in one and several complex variables.

For background information about this inequality and further work related to Bohr's inequality, we refer the reader to the recent survey by Abu-Muhanna et al. \cite{Abu-Muhanna-Ali-Ponnusamy-2016}
and the references therein.\vspace{1.2mm}

Renewed interest in Bohr's phenomenon surged in the 1990s, driven by successful extensions to holomorphic functions of several complex variables and more abstract functional-analytic settings. Notably, in 1997, Boas and Khavinson \cite{Boas-Khavinson-PAMS-1997} introduced and determined the $n$-dimensional Bohr radius for the family of bounded holomorphic functions on the unit polydisk. This seminal work stimulated a resurgence of research interest in Bohr-type questions across diverse mathematical domains. Subsequent investigations have yielded profound insights into the Bohr phenomenon for multidimensional power series, with notable contributions by Aizenberg et al. \cite{Aizenberg-PAMC-1999}-\cite{Aizenberg},
 Defant and Frerick \cite{Defant-Frerick-IJM-2006}, and Djakov and Ramanujan \cite{Djakov-Ramanujan-JA-2000}.
The Bohr radius has been studied extensively across various classes of functions, including locally univalent harmonic mappings, $k$-quasiconformal mappings, bounded harmonic functions, and lacunary series (see, for instance, \cite{Kayumov-Ponnusamy-Shakirov-2018, Kayumov-Ponnusamy-2018}). For a deeper exploration of the intriguing aspects and recent developments in the Bohr phenomenon for the complex-valued harmonic and pluriharmonic mappings, we refer the reader to \cite{Chen-Hamada-JFA-2022,Hamada-Honda-Mizota-2020,Liu-Ponnusamy-2023,Liu- Ponnusamy-Wang-2020,Liu-Ponnusamy-2019,Muhanna-CVEE-2010,Muhanna-Ali-Ng-Hansi-2014} and the references therein.

The absolute value of a complex number $z_1$ is denoted by $|z_1|$ and for $z=(z_1,\ldots,z_n)\in\mathbb{C}^n$, we define
\begin{align*}
	||z||^2=\sum\limits_{k=1}^n|z_k|^2\quad \text{and}\quad ||z||_{\infty}=\max\limits_{1\leq i\leq n}|z_i|.
\end{align*}

The Bohr phenomenon for pluriharmonic functions $f=h+\ol g\in \mathcal{H}_n^0$ of the form \eqref{Eq 1.7} is as follows: find the largest radius ${\bf{r}_f}\in (0,1)$ such that the following inequality
\begin{align}\label{Eq 1.8} 
M_f(n{\bf{r}}):=&\left|\sum\limits_{j=1}^n z_j\right|+\sum\limits_{m=2}^{\infty}\sum\limits_{|\beta|=m}\left(\left|\sum\limits_{|\beta|=m}a_{\beta}\right|+\left|\sum\limits_{|\beta|=m}b_{\beta}\right|\right)r^{\beta}\\\leq& d\left(f(0),\partial f\left(\mathbb{P} \Delta(0;1)\right)\right)\nonumber
\end{align}
holds for all $z=(z_1,z_2,\ldots,z_n)\in \mathbb{P} \Delta(0;1)$ such that $\|z\|_{\infty}={\bf{r}}\leq {\bf{r}_f}$, where $r=(r_1,r_2,\ldots,r_n)$ and for all $f\in\mathcal{H}_n^0$, where $M_f(n{\bf{r}})$ is the Majorant series of $f$. Such radius ${\bf{r}_f}$ is called the Bohr radius for the class $\mathcal{H}_n^0$. 

Using Theorems B and C, Allu and  Halder \cite{Allu-Halder-2021} obtained the Bohr radius for the class $\mathscr{P}_{\mathcal{H}}^{0}(\alpha)$.
\begin{theoE}\emph{\cite[Theorem 2.2]{Allu-Halder-2021}}
	Let $f\in\mathcal{P}_{H}^{0}(\alpha)$ be given by \emph{(\ref{eq.1})} with
	$0\le \alpha<1$. Then the inequality
\begin{align*}
M_f(r):=|z|+\sum_{n=2}^{\infty}\left(|a_n|+|b_n|\right)|z|^n
\leq d\bigl(f(\mathbf{0}),\partial f(\mathbb{D})\bigr)
\end{align*}
holds for $|z|=r\le r_f$, where $r_f$ is the unique positive root of
\begin{align*}
	r+2(1-\alpha)\sum_{m=2}^{\infty}\frac{r^{m}}{m}
	=
	1+2(1-\alpha)\sum_{m=2}^{\infty}\frac{(-1)^{m-1}}{m}
\end{align*}
in $(0,1)$. The radius $r_f$ is the best possible.
\end{theoE}

\subsection{\bf {Our aim}}	
At first, we now introduce the class $\mathscr{P}_{\mathcal{\mathcal{H}}_n^0}(\alpha)$ in $\mathbb{P} \Delta(0;1)$ as follows.
\begin{defi} For $\alpha$ with $0\leq \alpha<1$ and $z\in \mathbb{P} \Delta(0;1)$, let
	\begin{align*}
		\mathscr{P}_{\mathcal{H}_n^0}(\alpha)=\left\lbrace h+\ol g\in \mathcal{H}_n^0:\sum\limits_{k=1}^{\infty} \Re\left(\frac{\partial h(z)}{\partial z_k}-\alpha\right)>\sum\limits_{k=1}^{\infty}\left|\frac{\partial g(z)}{\partial z_k}\right| \right\rbrace.
	\end{align*}
\end{defi}

When $\dim(\mathbb{C}^n)=1$, we denote $\mathscr{P}_{\mathcal{H}_1^0}(\alpha)$ by $\mathscr{P}_{\mathcal{H}}^0(\alpha)$. Also we introduce the class $\mathscr{P}_{\mathcal{\mathcal{H}}_n^0}$ in $\mathbb{P} \Delta(0;1)$ as follows.
	\begin{align*}
		\mathscr{P}_{\mathcal{H}_n^0}=\left\lbrace h+\ol g\in \mathcal{H}_n^0:\sum\limits_{k=1}^{\infty} \Re\left(\frac{\partial h(z)}{\partial z_k}\right)>\sum\limits_{k=1}^{\infty}\left|\frac{\partial g(z)}{\partial z_k}\right| \right\rbrace.
	\end{align*}
\vspace{1.2mm}

It is natural to raise the following open problem.
\begin{prob}\label{P-1}
	Can we establish the multidimensional versions of Theorems A, B, C and E?
\end{prob}

The primary goal of this article is to provide an affirmative answer to Problem \ref{P-1}. 
The novelty of this research lies in the extension of the theory of \emph{harmonic} functions in $\mathbb{C}$ to \emph{pluriharmonic} functions in higher-dimensional complex spaces ($\mathbb{C}^n$). In this paper, we study the analytic and geometric properties for the class $\mathscr{P}_{\mathcal{H}_n^0}(\alpha)$.
The sharp coefficient bounds and growth estimates provided for the class $\mathscr{P}_{\mathcal{H}_n^0}(\alpha)$ represent a major advancement in the study of geometric function theory for several complex variables. Here we also obtain the Bohr radius for the class $\mathscr{P}_{\mathcal{H}_n^0}(\alpha)$ satisfying the inequality \eqref{Eq 1.8}. \vspace{1.2mm}

The paper is organized as follows: In Section \ref{Sec-2}, we obtain the bounds on the Taylor coefficients of $h$ and $g$ for the function $f=h+\ol g\in\mathscr{P}_{\mathcal{H}_n^0}(\alpha)$. In Section \ref{Sec-3}, we provide the Bohr radius for the class $\mathscr{P}_{\mathcal{H}_n^0}(\alpha)$ satisfying the inequality \eqref{Eq 1.8}. Also in Section \ref{Sec-4}, we study the partial sums of functions in $\mathscr{P}_{\mathcal{H}_n^0}(\alpha)$.

\section{{\bf Bounds on the Taylor coefficients of $h$ and $g$ for the function $f=h+\ol g\in\mathscr{P}_{\mathcal{H}_n^0}(\alpha)$}}\label{Sec-2}
The following result provides the sharp coefficient bounds for functions in $\mathscr{P}_{\mathcal{H}_n^0}(\alpha)$.

\begin{theo}\label{Th-1.2} Let $f=h+\ol g\in \mathscr{P}_{\mathcal{H}_n^0}(\alpha)$ and be given by \emph{(\ref{Eq 1.7})}. Then for any multi-index $\beta=(\beta_1,\beta_2,\ldots,\beta_n)$ such that $|\beta|=m\geq 2$, we have
\begin{align*}
\sum\limits_{|\beta|=m}|b_{\beta}|\leq n\binom{m+n-1}{n-1}\frac{1-\alpha}{m}
\end{align*}
The inequality is sharp.
\end{theo}
\begin{proof} Let $f=h+\ol g\in \mathscr{P}_{\mathcal{H}_n^0}(\alpha)$. Then
\begin{align}\label{Th2:1.1}
\sum\limits_{k=1}^{n}\Re\left(\frac{\partial h(z)}{\partial z_k}-\alpha\right)>\sum\limits_{k=1}^n\left|\frac{\partial g(z)}{\partial z_k}\right|
\end{align}
for all $z=(z_1,z_2,\ldots,z_n)\in \mathbb{P} \Delta(0_n;1_n)$. It follows from (\ref{Eq 1.7}) that
\begin{align}\label{Th2:1.2}
h(z)=\sum\limits_{j=1}^n z_j+\sum\limits_{m=2}^{\infty}P_m(z)\quad {and} \quad g(z)=\sum\limits_{m=2}^{\infty} Q_m(z)
\end{align}
for all $z\in \mathbb{P} \Delta(0_n;1_n)$, where
\begin{align}\label{Th2:1.3}
P_{m}(z)=\sum\limits_{|\beta|=m} a_{\beta} z^{\beta}\quad \text{and}\quad Q_m(z)=\sum\limits_{|\beta|=m} b_{\beta} z^{\beta}
\end{align}
are homogeneous polynomials of degree $m\geq 2$ in $z\in \mathbb{P} \Delta(0_n;1_n)$. We know that the number of terms in $\sum\limits_{|\beta|=m}$ is $\binom{|\beta|+n-1}{n-1}$. A simple computation using (\ref{Th2:1.3}) shows that
\begin{align}\label{Th2:1.3a}
\frac{\partial Q_m(z)}{\partial z_k}=\sum\limits_{|\beta|=m}\beta_{k} b_{\beta}z_1^{\beta_1}\ldots z_{j-1}^{\beta_{j-1}}z_{j}^{\beta_j}z_{j+1}^{\beta_{j+1}}\ldots z_{k-1}^{\beta_{k-1}}z_{k}^{\beta_{k}-1}z_{k+1}^{\beta_{k+1}}\ldots z_n^{\beta_n}
\end{align}

Using (\ref{Th2:1.2}) together with (\ref{Th2:1.3a}), we get
\begin{align}\label{Th2:1.5a}
\frac{\partial g(z)}{\partial z_k}&=
\sum\limits_{m=2}^{\infty} \frac{\partial Q_m(z)}{\partial z_k}\\=&\sum\limits_{m=2}^{\infty}\sum\limits_{|\beta|=m}\beta_{k} b_{\beta}z_1^{\beta_1}\ldots z_{j-1}^{\beta_{j-1}}z_{j}^{\beta_j}z_{j+1}^{\beta_{j+1}}\ldots z_{k-1}^{\beta_{k-1}}z_{k}^{\beta_{k}-1}z_{k+1}^{\beta_{k+1}}\ldots z_n^{\beta_n}.\nonumber
\end{align}

By applying Cauchy's integral formula to $\frac{\partial g(z)}{\partial z_k}$, it follows from \eqref{Th2:1.5a} that
\begin{align}\label{Th2:1.6a}
&(2\pi i)^n \beta_{k}b_{\beta}\\=&\int\limits_{|z_1|=r_1}\ldots \int\limits_{|z_n|=r_n}\frac{\frac{\partial g(z)}{\partial z_k}\;dz_1\;d z_2\ldots d z_n}{z_1^{\beta_1+1}\ldots z_{j-1}^{\beta_{j-1}+1}z_{j}^{\beta_j+1}z_{j+1}^{\beta_{j+1}+1}\ldots z_{k-1}^{\beta_{k-1}+1}z_{k}^{\beta_{k}}z_{k+1}^{\beta_{k+1}+1}\ldots z_n^{\beta_n+1}}.\nonumber
\end{align}
where $0<r_j<1$ for $j=1,2,\ldots,n$. We set $z=\left(re^{\iota\theta_1},re^{\iota\theta_2},\ldots,re^{\iota\theta_n}\right)$, where $0\leq \theta_j\leq 2\pi$, $j=1,2,\ldots,n$.
Now from (\ref{Th2:1.6a}), we have
\begin{align}\label{Th2:1.6b}
(2\pi i)^n \beta_{k}|b_{\beta}|\leq \int\limits_{0}^{2\pi}\ldots \int\limits_0^{2\pi}\frac{\left|\frac{\partial g(z)}{\partial z_k}\right|\;d\theta_1 d\theta_2\ldots d\theta_n}{r^{m-1}}
\end{align}
and so
\begin{align}\label{Th2:1.7}
(2\pi)^n\sum\limits_{|\beta|=m}\beta_k|b_{\beta}|\leq \binom{|\beta|+n-1}{n-1}\int\limits_{0}^{2\pi}\ldots \int\limits_0^{2\pi}\frac{\left|\frac{\partial g(z)}{\partial z_k}\right|\;d\theta_1d\theta_2\ldots d\theta_n}{r^{m-1}}
\end{align}
for all $k=1,2,\ldots,n$. Consequently from (\ref{Th2:1.7}), we obtain
\begin{align}\label{Th2:1.8}
(2\pi)^nr^{m-1} \sum\limits_{k=1}^n\sum\limits_{|\beta|=m} \beta_k|b_{\beta}|=&
(2\pi)^nr^{m-1} \sum\limits_{|\beta|=m}|b_{\beta}|\sum\limits_{k=1}^n \beta_k\\ \leq &\binom{|\beta|+n-1}{n-1}\int\limits_{0}^{2\pi}\ldots \int\limits_0^{2\pi}\sum\limits_{k=1}^n\left|\frac{\partial g(z)}{\partial z_k}\right|\;d\theta_1d\theta_2\ldots d\theta_n.\nonumber
\end{align}
However, we observe that 
\begin{align}\label{Th2:1.8a}
\sum\limits_{|\beta|=m}|b_{\beta}|\sum\limits_{k=1}^n\beta_{k}=m\sum\limits_{|\beta|=m}|b_{\beta}|.
\end{align}

Using (\ref{Th2:1.8a}) together with (\ref{Th2:1.8}), we get
\begin{align}\label{Th2:1.8b}
(2\pi)^nr^{m-1}m \sum\limits_{|\beta|=m}|b_{\beta}|
 \leq \binom{|\beta|+n-1}{n-1}\int\limits_{0}^{2\pi}\ldots \int\limits_0^{2\pi}\sum\limits_{k=1}^n\left|\frac{\partial g(z)}{\partial z_k}\right|\;d\theta_1 d\theta_2\ldots d\theta_n.
\end{align}

Moreover, we know that
\begin{align}\label{Th2:1.8c}
&\frac{\partial h(z)}{\partial z_k}\\=&1+
\sum\limits_{m=2}^{\infty}\sum\limits_{|\beta|=m}\beta_{k} a_{\beta}z_1^{\beta_1}\ldots z_{j-1}^{\beta_{j-1}}z_{j}^{\beta_j}z_{j+1}^{\beta_{j+1}}\ldots z_{k-1}^{\beta_{k-1}}z_{k}^{\beta_{k}-1}z_{k+1}^{\beta_{k+1}}\ldots z_n^{\beta_n}.\nonumber
\end{align}

For any multi-index $\nu=(\nu_1,\nu_2,\ldots,\nu_n)$, we have
\begin{align}\label{Th2:1.9}
\int\limits_{0}^{2\pi}\ldots \int\limits_0^{2\pi} (e^{i\theta_1})^{\nu_1}\ldots (e^{i\theta_n})^{\nu_n}d\theta_1 \ldots d\theta_n=
\begin{cases}
0,& \nu\neq (0,0,\ldots,0),\\[2ex]
(2\pi)^n,& \nu=(0,0,\ldots,0).
\end{cases}
\end{align}

Using (\ref{Th2:1.9}) together with (\ref{Th2:1.8c}), we deduce that
\begin{align*}
\frac{1}{(2\pi)^n}\int\limits_{0}^{2\pi}\ldots \int\limits_0^{2\pi}\left(\frac{\partial h(z)}{\partial z_k}-\alpha\right)\;d\theta_1\;d \theta_2\ldots d \theta_n=1-\alpha
\end{align*}
and so we get
\begin{align*}
\frac{1}{(2\pi)^n}\int\limits_{0}^{2\pi}\ldots \int\limits_0^{2\pi}\Re\left(\frac{\partial h(z)}{\partial z_k}-\alpha\right)\;d\theta_1\;d \theta_2\ldots d \theta_n=1-\alpha.
\end{align*}
Consequently
\begin{align}\label{Th2:1.9a}
\frac{1}{(2\pi)^n}\int\limits_{0}^{2\pi}\ldots \int\limits_0^{2\pi} \sum\limits_{k=1}^n\Re\left(\frac{\partial h(z)}{\partial z_k}-\alpha\right)\;d\theta_1\;d \theta_2\ldots d \theta_n=n(1-\alpha).
\end{align}
In view of (\ref{Th2:1.1}) and (\ref{Th2:1.9a}) and using (\ref{Th2:1.8b}), we obtain
\begin{align*}
&r^{m-1}m\sum\limits_{|\beta|=m}|b_{\beta}|
\\ \leq &\binom{|\beta|+n-1}{n-1}\frac{1}{(2\pi)^n}\int\limits_{0}^{2\pi}\ldots \int\limits_0^{2\pi}\sum\limits_{k=1}^n\left|\frac{\partial g(z)}{\partial z_k}\right|\;d\theta_1\;d \theta_2\ldots d \theta_n\\\leq &
\binom{|\beta|+n-1}{n-1}\frac{1}{(2\pi)^n}\int\limits_{0}^{2\pi}\ldots \int\limits_0^{2\pi}\sum\limits_{k=1}^n\Re\left(\frac{\partial h(z)}{\partial z_k}-\alpha\right)\;d\theta_1\;d \theta_2\ldots d \theta_n\\\leq&n
\binom{|\beta|+n-1}{n-1}(1-\alpha).
\end{align*}

Letting $r\to 1^-$, we obtain
\begin{align*}
\sum\limits_{|\beta|=m}|b_{\beta}|\leq n\binom{m+n-1}{n-1}\frac{1-\alpha}{m}
\end{align*}
if $m\geq 2$.\vspace{1.2mm}

To show that the bound is sharp, we consider the following function
\begin{align*}
f_1(z)=\sum\limits_{j=1}^n z_j+\sum\limits_{m=2}^{\infty}\sum\limits_{|\beta|=m} b_{\beta} \ol{z^{\beta}},
\end{align*}
where $b_{\beta}=\frac{n(1-\alpha)}{m}$
for all multi-index $\beta=(\beta_1,\beta_2,\ldots,\beta_n)$ such that $|\beta|=m$. 
A direct computation shows that $f_1\in \mathscr{P}_{\mathcal{H}_n^0}(\alpha)$, and 
\begin{align*}
\sum\limits_{|\beta|=m}|b_{\beta}(f_1)|=n\binom{m+n-1}{n-1}\frac{1-\alpha}{m}.
\end{align*} 
\end{proof}

\begin{lem}\label{Lm-3.1}\emph{\cite[Theorem 6.1.4]{Graham-Kohr}} Let $f$ be holomorphic in the polydisk $\mathbb{P}\Delta(0_n;1_n)$ such that $|f(z)|\leq 1$ for all $z\in \mathbb{P}\Delta(0_n;1_n)$. Then
\begin{align*}
\left|\frac{\partial^{|\beta|} f(0)}{\partial z_1^{\beta_1}\ldots \partial z_n^{\beta_n}}\right|\leq \beta!
\end{align*}
 for multi-index $\beta=(\beta_1,\ldots, \beta_n)$.
\end{lem}

Let $G\not=\varnothing$ be an open subset of $\mathbb{C}^n$. Let $f$ be a holomorphic function on $G$. For a point $a\in\mathbb{C}^n$, we write 
\begin{align*}
	f(z)=\sum_{i=0}^{\infty}P_i(z-a),
\end{align*}
where the term $P_i(z-a)$ is either identically zero or a homogeneous polynomial of degree $i$. Denote the zero-multiplicity of $f$ at $a$ by 
\begin{align*}
	k=\min\{i:P_i(z-a)\not\equiv 0\}.
\end{align*}
Clearly $1$ is the zero-multiplicity of $f$ at $a$ when $f(a)=0$ and $\frac{\partial f(a)}{\partial z_j}\neq 0$ for some $j=1,2,\ldots,n$.

Now we state the multidimensional version of Theorem C ($(i)$ and $(ii)$).
\begin{theo}\label{Th-1.3} Let $f=h+\ol g\in \mathscr{P}_{\mathcal{H}_n^0}(\alpha)$ and be given by \emph{(\ref{Eq 1.7})}. Then for any multi-index $\beta=(\beta_1,\beta_2,\ldots,\beta_n)$ such that $|\beta|=m\geq 2$, we have
\begin{enumerate}
\item[\emph{(i)}] $\displaystyle \left|\sum\limits_{|\beta|=m}a_{\beta}\right|+\left|\sum\limits_{|\beta|=m}b_{\beta}\right|\leq 2\binom{m+n-1}{n-1}\frac{n(1-\alpha)}{m}$,\vspace{1.2mm}
\item[\emph{(ii)}] $\displaystyle \left|\;\left|\sum\limits_{|\beta|=m}a_{\beta}\right|-\left|\sum\limits_{|\beta|=m}b_{\beta}\right|\;\right|\leq 2\binom{m+n-1}{n-1}\frac{n(1-\alpha)}{m}$,\vspace{1.2mm}
\item[\emph{(iii)}] 
$\displaystyle \left|\sum\limits_{|\beta|=m}a_{\beta}\right|\leq 2\binom{m+n-1}{n-1}\frac{n(1-\alpha)}{m}$.
\end{enumerate}
All three inequalities are sharp.
\end{theo}

\begin{proof}Let $f=h+\ol g\in \mathscr{P}_{\mathcal{H}_n^0}(\alpha)$. Note that 
\begin{align*}
\displaystyle \sum\limits_{k=1}^n\Re\left(\frac{\partial h(z)}{\partial z_k}+\varepsilon \frac{\partial g(z)}{\partial z_k}-\alpha\right)=&
\sum\limits_{k=1}^n\Re\left(\frac{\partial h(z)}{\partial z_k}-\alpha\right)+\sum\limits_{k=1}^n\Re\left(\varepsilon \frac{\partial g(z)}{\partial z_k}\right)\\>&
\sum\limits_{k=1}^n\Re\left(\frac{\partial h(z)}{\partial z_k}-\alpha\right)-\sum\limits_{k=1}^n\left|\varepsilon \frac{\partial g(z)}{\partial z_k}\right|\\=&
\sum\limits_{k=1}^n\Re\left(\frac{\partial h(z)}{\partial z_k}-\alpha\right)-\sum\limits_{k=1}^n\left|\frac{\partial g(z)}{\partial z_k}\right|
\end{align*}
for all $\varepsilon\in\mathbb{C}$ such that $|\varepsilon|=1$. Therefore from \eqref{Th2:1.1}, we deduce that 
\begin{align}\label{Th3:1.1}
\displaystyle\Re\left(\sum\limits_{k=1}^n\left(\frac{\partial h(z)}{\partial z_k}+\varepsilon \frac{\partial g(z)}{\partial z_k}\right)\right)-n\alpha=\sum\limits_{k=1}^n\Re\left(\frac{\partial h(z)}{\partial z_k}+\varepsilon \frac{\partial g(z)}{\partial z_k}-\alpha\right)>0
\end{align}
for all $z=(z_1,z_2,\ldots,z_n)\in \mathbb{P} \Delta(0_n;1_n)$. Let $F_{\varepsilon}(z)=h(z)+\varepsilon g(z)$. Clearly
\begin{align*}
\displaystyle F_{\varepsilon}(z)=\sum\limits_{j=1}^n z_j+\sum\limits_{m=2}^{\infty}\left(P_m(z)+\varepsilon Q_m(z)\right),
\end{align*}
for all $z\in \mathbb{P} \Delta(0_n;1_n)$, where $P_m(z)$ and $Q_m(z)$ are defined in (\ref{Th2:1.3}). Note that
 \begin{align*}
 \displaystyle \frac{\partial F_{\varepsilon}(z)}{\partial z_k}=&1+
 \sum\limits_{m=2}^{\infty} \frac{\partial \left(P_m(z)+\varepsilon Q_m(z)\right)}{\partial z_k}\\=&1+\sum\limits_{m=2}^{\infty}\sum\limits_{|\beta|=m}\beta_{k} c_{\beta}z_1^{\beta_1}\ldots z_{j-1}^{\beta_{j-1}}z_{j}^{\beta_j}z_{j+1}^{\beta_{j+1}}\ldots z_{k-1}^{\beta_{k-1}}z_{k}^{\beta_{k}-1}z_{k+1}^{\beta_{k+1}}\ldots z_n^{\beta_n},
 \end{align*}
where 
$c_{\beta}=a_{\beta}+\varepsilon b_{\beta}$.
Consequently
\begin{align}\label{Th3:1.2}
\displaystyle &\sum\limits_{k=1}^n\frac{\partial F_{\varepsilon}(z)}{\partial z_k}\\ =&n+
\sum\limits_{m=2}^{\infty}\sum\limits_{|\beta|=m}\sum\limits_{k=1}^n\beta_{k} c_{\beta}z_1^{\beta_1}\ldots z_{j-1}^{\beta_{j-1}}z_{j}^{\beta_j}z_{j+1}^{\beta_{j+1}}\ldots z_{k-1}^{\beta_{k-1}}z_{k}^{\beta_{k}-1}z_{k+1}^{\beta_{k+1}}\ldots z_n^{\beta_n}\nonumber
\end{align}
for all $z=(z_1,z_2,\ldots,z_n)\in \mathbb{P} \Delta(0_n;1_n)$. Then from (\ref{Th3:1.1}), we see that
\begin{align*} \Re\left(\sum\limits_{k=1}^n\frac{\partial F_{\varepsilon}(z)}{\partial z_k}\right)-n\alpha>0\;\;\text{and}\;\;\sum\limits_{k=1}^n\frac{\partial F_{\varepsilon}(0)}{\partial z_k}-n\alpha=n(1-\alpha).
\end{align*}
Therefore, there exists a function \begin{align}\label{Th3:1.3}
	P(z)=1+\sum\limits_{m=1}^{\infty}\sum\limits_{|\beta|=m}P_{\beta}z^{\beta}
\end{align}
 such that
\begin{align}\label{Th3:1.4}
\sum\limits_{k=1}^n\frac{\partial F_{\varepsilon}(z)}{\partial z_k}=n+n(1-\alpha)\sum\limits_{m=1}^{\infty}\sum\limits_{|\beta|=m}P_{\beta}z^{\beta}
\end{align}
for all $z\in \mathbb{P}\Delta(0_n;1_n)$, where $P(z)$ is holomorphic in $\mathbb{P}\Delta(0_n;1_n)$ and $\Re(P(z))>0$ in $\mathbb{P}\Delta(0_n;1_n)$. 

Let $\omega_l=e^{\frac{2\pi \iota}{k}l}$, where $k\geq 1$ is an integer. We see that $\sum_{l=1}^k \omega_l^s=0$ if $s\leq k-1$. Let
\begin{align*}
	\tilde g(z)=\frac{1}{k}\sum\limits_{l=1}^k P\left(\omega_l z\right).
\end{align*}

The function $\tilde g$ is clearly holomorphic in $\mathbb{P}\Delta(0_n;1_n)$ such that $\Re\{\tilde g(z)\} > 0$ throughout the domain. Furthermore, we see that $\tilde g(0) =1$ and $\tilde g$ has the series expansion
\begin{align*}
	\tilde g(z)=1+\sum\limits_{m=1}^{\infty}\sum\limits_{|\beta|=mk} P_{\beta}z^{\beta}
\end{align*}
for all $z\in \mathbb{P}\Delta(0_n;1_n)$. Let
\begin{align}
	\label{Eq-3.20}\phi(z)=\frac{\tilde g(z)-1}{\tilde g(z)+1}
\end{align}
for all $z\in \mathbb{P}\Delta(0_n;1_n)$. One can easily verify that $|\phi(z)|<1$ for all $z\in \mathbb{P}\Delta(0_n;1_n)$ and $k$ is the zero-multiplicity of $\phi$ at $0$. We expand $\phi(z)$ to a Taylor series with multi-index $\beta$,
\begin{align}
	\label{Eq-3.21} \phi(z)=\sum\limits_{m=k}^{\infty}\sum\limits_{|\beta|=m}Q_{\beta}z^{\beta}
\end{align}
By Lemma \ref{Lm-3.1}, we have $|Q_{\beta}|\leq 1$
%\begin{align}
%\label{Eq-3.22} |b_{\beta}|\leq 1
%\end{align}
for multi-index $\beta$. Now from \eqref{Eq-3.20} and \eqref{Eq-3.21}, we deduce that
\begin{align*}
	\sum\limits_{|\beta|=k} P_{\beta}z^{\beta}+\sum\limits_{m=2}^{\infty}\sum\limits_{|\beta|=mk} P_{\beta}z^{\beta}&=2 \sum\limits_{m=k}^{\infty}\sum\limits_{|\beta|=m}Q_{\beta}z^{\beta}\\&+
	\left(\sum\limits_{|\beta|=k} P_{\beta}z^{\beta}+\sum\limits_{m=2}^{\infty}\sum\limits_{|\beta|=mk} P_{\beta}z^{\beta}\right)\left(\sum\limits_{m=k}^{\infty}\sum\limits_{|\beta|=m}Q_{\beta}z^{\beta}\right).
\end{align*}
It then follows that
\begin{align}
	\label{Eq-3.23} \sum\limits_{|\beta|=k} \left(P_{\beta}-2Q_{\beta}\right)z^{\beta}+\text{higher degree terms of}\; (z_1^{\beta_1},\ldots,z_n^{\beta_n})\equiv 0
\end{align}
holds for all $z\in \mathbb{P}\Delta(0_n;1_n)$. Comparing the homogeneous terms of degree $k$ in
\eqref{Eq-3.23},
we obtain 
\begin{align}
	\label{Eq-3.24}\sum\limits_{|\beta|=k} \left(P_{\beta}-2Q_{\beta}\right)=0.
\end{align}
\noindent Since $|Q_{\beta}|\leq 1$ for multi-index $\beta$, from \eqref{Eq-3.24}, we conclude that
\begin{align}\label{Th3:1.5}
	\sum\limits_{|\beta|=k} |P_{\beta}|\leq 2\sum\limits_{|\beta|=k} |Q_{\beta}|\leq 2\sum\limits_{|\beta|=k} 1\leq 2\binom{k+n-1}{n-1},
\end{align}
where $\beta=(\beta_1,\beta_2,\ldots,\beta_n)$ such that $|\beta|=k$. \vspace{1.2mm}

On the other hand, from (\ref{Th3:1.2}) and (\ref{Th3:1.4}), we obtain
\begin{align}\label{Th3:1.6}
&\sum\limits_{m=2}^{\infty}\sum\limits_{|\beta|=m}\sum\limits_{k=1}^n\beta_{k} c_{\beta}z_1^{\beta_1}\ldots z_{j-1}^{\beta_{j-1}}z_{j}^{\beta_j}z_{j+1}^{\beta_{j+1}}\ldots z_{k-1}^{\beta_{k-1}}z_{k}^{\beta_{k}-1}z_{k+1}^{\beta_{k+1}}\ldots z_n^{\beta_n}\\=&n(1-\alpha)\sum\limits_{m=1}^{\infty}\sum\limits_{|\beta|=m}P_{\beta}z^{\beta}\nonumber
\end{align}
where $\beta=(\beta_1,\beta_2,\ldots,\beta_n)$ such that $|\beta|=m\geq 2$. Putting $z_1=z_2=\ldots=z_n=z$ on both sides of \eqref{Th3:1.6} and then comparing coefficients, we get
\begin{align*}
\sum\limits_{|\beta|=m}\sum\limits_{k=1}^n
\beta_kc_{\beta}=n(1-\alpha)\sum\limits_{|\beta|=m}P_{\beta}
\end{align*}
and so from \eqref{Th3:1.5}, we get
\begin{align}\label{Th3:1.7}
\left|\sum\limits_{|\beta|=m}a_{\beta}+\varepsilon\sum\limits_{|\beta|=m} b_{\beta}\right|\leq 2\binom{m+n-1}{n-1}\frac{n(1-\alpha)}{m}
\end{align}
for all $\beta=(\beta_1,\beta_2,\ldots,\beta_n)$ such that $|\beta|=m\geq 2$ and for all $\varepsilon$ such that $|\varepsilon|=1$. 
	
Suppose $A_{\beta}=\sum\limits_{|\beta|=m}a_\beta\neq 0$ and $B_{\beta}=\sum\limits_{|\beta|=m}b_\beta\neq 0$. If we choose 
\begin{align*}
\varepsilon=\frac{A_{\beta}}{|A_{\beta}|}\frac{\ol B_{\beta}}{|B_{\beta}|},
\end{align*}
then $|\varepsilon|=1$ and so from (\ref{Th3:1.7}), we have
\begin{align}\label{Th3:1.8}
|A_{\beta}+\varepsilon B_{\beta}|=|A_{\beta}|+|B_{\beta}|\leq 2\binom{m+n-1}{n-1}\frac{n(1-\alpha)}{m}
\end{align}
for $\beta=(\beta_1,\beta_2,\ldots,\beta_n)$ such that $|\beta|=m\geq 2$. If either $A_{\beta}=0$ or $B_{\beta}=0$, then (\ref{Th3:1.8}) also holds. Therefore, we have
\begin{align}\label{Th3:1.8a}
\left|\sum\limits_{|\beta|=m}a_{\beta}\right|+\left|\sum\limits_{|\beta|=m}b_{\beta}\right|\leq 2\binom{m+n-1}{n-1}\frac{n(1-\alpha)}{m}
\end{align}
for $\beta=(\beta_1,\beta_2,\ldots,\beta_n)$ such that $|\beta|=m\geq 2$.

On the other hand, we have
\begin{align*}
\left||A_{\beta}|-|B_{\beta}|\right|\leq |A_{\beta}+\varepsilon B_{\beta}|\leq 2\binom{m+n-1}{n-1}\frac{n(1-\alpha)}{m}
\end{align*}
for $\beta=(\beta_1,\beta_2,\ldots,\beta_n)$ such that $|\beta|=m\geq 2$. Consequently
\begin{align*}
\left|\;\left|\sum\limits_{|\beta|=m}a_{\beta}\right|-\left|\sum\limits_{|\beta|=m}b_{\beta}\right|\;\right|\leq 2\binom{m+n-1}{n-1}\frac{n(1-\alpha)}{m}
\end{align*}
for $\beta=(\beta_1,\beta_2,\ldots,\beta_n)$ such that $|\beta|=m\geq 2$.

It is clear from \eqref{Th3:1.8a} that 
\begin{align*}
\left|\sum\limits_{|\beta|=m}a_{\beta}\right|\leq 2\binom{m+n-1}{n-1}\frac{n(1-\alpha)}{m},
\end{align*}
where $\beta=(\beta_1,\beta_2,\ldots,\beta_n)$ such that $|\beta|=m\geq 2$.\vspace{1.2mm}

To show all the inequalities are sharp, we consider the following function
\begin{align}\label{Ex1}
f_2(z)=\sum\limits_{j=1}^n z_j+\sum\limits_{m=2}^{\infty}\sum\limits_{|\beta|=m} a_{\beta} z^{\beta},
\end{align}
where 
\begin{align*}
a_{\beta}=\frac{2n(1-\alpha)}{m}
\end{align*}
for all multi-index $\beta=(\beta_1,\beta_2,\ldots,\beta_n)$ such that $|\beta|=m$. 
A direct computation shows that $f_2\in \mathscr{P}_{\mathcal{H}_n^0}(\alpha)$, and 
\begin{align*}
\sum\limits_{|\beta|=m}|a_{\beta}(f_2)|=2\binom{m+n-1}{n-1}\frac{n(1-\alpha)}{m}.
\end{align*} 
\end{proof}

\section{{\bf Growth estimates for functions in $\mathscr{P}_{\mathcal{H}_n^0}(M)$}}\label{Sec-3}

In the following result, we have established the growth estimates for the functions in the class $\mathscr{P}_{\mathcal{H}_n^0}(\alpha)$.

\begin{theo}\label{Th-1.4}Let $f=h+\ol g\in \mathscr{P}_{\mathcal{H}_n^0}(\alpha)$ with $0<\alpha\leq 1$ and be given by \emph{(\ref{Eq 1.7})}. Then for $z\in \mathbb{P} \Delta(0_n;1_n)$, we have
\begin{align*}
	\left|\sum\limits_{j=1}^nz_j\right|+2n(1-\alpha)\sum\limits_{m=2}^{\infty}
	(-1)^{m-1}\frac{\|z\|_{\infty}^m}{m}\leq |f(z)|\leq n\|z\|_{\infty}+2n(1-\alpha)\sum\limits_{m=2}^{\infty}
	\frac{\|z\|_{\infty}^m}{m}.
\end{align*}
 Both the inequalities are sharp.
\end{theo}

\begin{proof} Let $f=h+\ol g\in \mathscr{P}_{\mathcal{H}_n^0}(\alpha)$. Let $F_{\varepsilon}=h+\varepsilon g$. Then from (\ref{Th3:1.1}), we see that
\begin{align*} \Re\left(\sum\limits_{k=1}^n\frac{\partial F_{\varepsilon}(z)}{\partial z_k}\right)-n\alpha>0\;\;\text{and}\;\;\sum\limits_{k=1}^n\frac{\partial F_{\varepsilon}(0)}{\partial z_k}-n\alpha=n(1-\alpha)
\end{align*}
for all $z=(z_1,z_2,\ldots,z_n)\in \mathbb{P} \Delta(0_n;1_n)$. Then there exists a holomorphic function $\omega(z)$ with $\omega(0)=0$ and $|\omega(z)|<1$ for all $z=(z_1,z_2,\ldots,z_n)\in \mathbb{P} \Delta(0_n;1_n)$ such that
\begin{align*}
\frac{\sum\limits_{k=1}^n\frac{\partial F_{\varepsilon}(z)}{\partial z_k}-n\alpha}{n(1-\alpha)}=\frac{1+\omega(z)}{1-\omega(z)}
\end{align*} 
or equivalently 
\begin{align}\label{Th4:1.1}	
	\sum\limits_{k=1}^n\frac{\partial F_{\varepsilon}(z)}{\partial z_k}=n\alpha+n(1-\alpha)\frac{1+\omega(z)}{1-\omega(z)}
\end{align}
for all $z=(z_1,z_2,\ldots,z_n)\in \mathbb{P} \Delta(0_n;1_n)$.	
Also by Lemma 6.1.28 \cite{Graham-Kohr}, we have
\begin{align}\label{Th4:1.1a}
|\omega(z)|\leq \|z\|_{\infty}
\end{align}
for all $z=(z_1,z_2,\ldots,z_n)\in \mathbb{P} \Delta(0_n;1_n)$.\vspace{1.2mm}

Let $z\in \mathbb{P}\Delta(0_n;1_n)$ be fixed in such a way that $z\neq 0$. Obviously $\|z\|_{\infty}<1$. 

We define 
\begin{align}\label{Th4:1.2}
	\tilde h(t)=h\left(\frac{z}{||z||_{\infty}}t\right)=\frac{1}{||z||_{\infty}}\left(\sum\limits_{j=1}^n z_j\right)t+\sum\limits_{m=2}^{\infty}P_m\left(\frac{z}{||z||_{\infty}}\right)t^m
\end{align}
and
\begin{align}\label{Th4:1.3}
	\tilde g(t)=g\left(\frac{z}{||z||_{\infty}}t\right)=\sum\limits_{k=2}^{\infty} Q_m\left(\frac{z}{||z||_{\infty}}\right)t^m,
\end{align}
where $h(z)$ and $g(z)$ are defined by (\ref{Th2:1.2}) and $t\in\mathbb{C}$ such that $|t|<1$. For a fixed $z\in \mathbb{P}\Delta(0_n;1_n)$, we can say that $\tilde g(t)$ and $\tilde h(t)$ are analytic in $|t|<1$. Obviously for a fixed $z\in \mathbb{P}\Delta(0_n;1_n)$, the following function
\begin{align}\label{Th4:1.4}
	\tilde F_{\varepsilon}(t)=F_{\varepsilon}\left(\frac{z}{||z||_{\infty}}t\right)=\tilde h(t)+\varepsilon \tilde g(t)
\end{align}
is analytic in $|t|<1$, where $\tilde h(t)$ and $\tilde g(t)$ are defined by (\ref{Th4:1.2}) and (\ref{Th4:1.3}). Also from (\ref{Th4:1.2}) and (\ref{Th4:1.3}), we see that
\begin{align}\label{Th4:1.5}
	\tilde f(t)=f\left(\frac{z}{||z||_{\infty}}t\right)=\tilde h(t)+\ol{\tilde g(t)}.
\end{align}

We define 
\begin{align}\label{Th4:1.6}
	\tilde \omega(t)=\omega\left(\frac{z}{||z||_{\infty}}t\right),
\end{align}
which is analytic in $|t|<1$. Consequently from \eqref{Th4:1.1a}, we have
\begin{align}\label{Th4:1.6a}
	|\tilde \omega(t)|\leq |t|.
\end{align}

In terms of one variable $t$, we deduce from (\ref{Th4:1.2}) that
\begin{align}\label{Th4:1.7}
	\frac{d \tilde F_{\varepsilon}(t)}{d t}=n\alpha+n(1-\alpha)\frac{1+\tilde \omega(t)}{1-\tilde \omega(t)}.
\end{align}
	
Note that
\begin{align}\label{Th4:1.8}
\tilde F_{\varepsilon}(t)=\tilde F_{\varepsilon}(t)-\tilde F_{\varepsilon}(0)=\int\limits_{0}^{t} \frac{d \tilde F_{\varepsilon}(s)}{d s} ds.
\end{align}

Now using  (\ref{Th4:1.6a}) and \eqref{Th4:1.7} into (\ref{Th4:1.8}), we obtain
\begin{align*}
\left|\tilde F_{\varepsilon}\left(\|z\|_{\infty}\right)\right|=
\left|\int\limits_{0}^{\|z\|_{\infty}}\frac{d \tilde F_{\varepsilon}(s)}{d s}ds\right|
\leq & 
\int\limits_{0}^{\|z\|_{\infty}}\left(n\alpha+n(1-\alpha)\left|\frac{1+\tilde \omega(\xi e^{i\theta})}{1-\tilde \omega(\xi e^{i\theta})}\right|\right)d \xi\\\leq & 
\int\limits_{0}^{\|z\|_{\infty}}\left(n\alpha+n(1-\alpha)\frac{1+\xi}{1-\xi}\right)d \xi
\nonumber\\ =&
\int\limits_{0}^{\|z\|_{\infty}} \left(n+2n(1-\alpha)\left(\xi+\xi^2+\xi^3+\ldots\right)\right)d\xi\nonumber\\= &
n\|z\|_{\infty}+2n(1-\alpha)\sum\limits_{m=2}^{\infty}
\frac{\|z\|_{\infty}^m}{m}\nonumber
\end{align*}
and so from \eqref{Th4:1.4}, we get
\begin{align}\label{Th4:1.9}
|F_{\varepsilon}(z)|\leq n\|z\|_{\infty}+2n(1-\alpha)\sum\limits_{m=2}^{\infty}
\frac{\|z\|_{\infty}^m}{m},
\end{align}
which also holds for $z=0$.

Again using (\ref{Th4:1.6a}) and \eqref{Th4:1.7} into (\ref{Th4:1.8}), we obtain
\begin{align*}
	\left|\tilde F_{\varepsilon}\left(\|z\|_{\infty}\right)\right|
	\geq & 
	\int\limits_{0}^{\|z\|_{\infty}}\left(n\alpha+n(1-\alpha)\Re\left(\frac{1+\tilde \omega(\xi e^{i\theta})}{1-\tilde \omega(\xi e^{i\theta})}\right)\right)d \xi\\\geq & 
	\int\limits_{0}^{\|z\|_{\infty}}\left(n\alpha+n(1-\alpha)\frac{1-\xi}{1+\xi}\right)d \xi
	\nonumber\\ =&
	\int\limits_{0}^{\|z\|_{\infty}} \left(n+2n(1-\alpha)\left(-\xi+\xi^2-\xi^3+\ldots\right)\right)d\xi\nonumber\\= &
	n\|z\|_{\infty}+2n(1-\alpha)\sum\limits_{m=2}^{\infty}
	(-1)^{m-1}\frac{\|z\|_{\infty}^m}{m}\nonumber
\end{align*}
and so from \eqref{Th4:1.4}, we get
\begin{align}\label{Th4:1.10}
	|F_{\varepsilon}(z)|\geq n\|z\|_{\infty}+2n(1-\alpha)\sum\limits_{m=2}^{\infty}
	(-1)^{m-1}\frac{\|z\|_{\infty}^m}{m},
\end{align}
which also holds for $z=0$.

Since $|\sum_{j=1}^nz_j|\leq n\|z\|_{\infty}$, from \eqref{Th4:1.9} and \eqref{Th4:1.10}, we get
\begin{align*}
\left|\sum\limits_{j=1}^nz_j\right|+2n(1-\alpha)\sum\limits_{m=2}^{\infty}
(-1)^{m-1}\frac{\|z\|_{\infty}^m}{m}\leq |F_{\varepsilon}(z)|\leq n\|z\|_{\infty}+2n(1-\alpha)\sum\limits_{m=2}^{\infty}
\frac{\|z\|_{\infty}^m}{m}.
\end{align*}

Finally since $\varepsilon\; (|\varepsilon|=1)$ is arbitrary and $F_{\varepsilon}(z)=h(z)+\varepsilon g(z)$, for each $0\leq \alpha<1$, we have
\begin{align*}
	\left|\sum\limits_{j=1}^nz_j\right|+2n(1-\alpha)\sum\limits_{m=2}^{\infty}
	(-1)^{m-1}\frac{\|z\|_{\infty}^m}{m}\leq |f(z)|\leq n\|z\|_{\infty}+2n(1-\alpha)\sum\limits_{m=2}^{\infty}
	\frac{\|z\|_{\infty}^m}{m}.
\end{align*}

Thus the desired inequalities are established.\vspace{2mm}
	
To show the left inequality is sharp, we consider the function $f_3(z)$ defined by
	\begin{align*}
		f_3(z)=\sum\limits_{j=1}^n z_j+\sum\limits_{m=2}^{\infty}\sum\limits_{|\beta|=m}a_{\beta}z^{\beta},
	\end{align*}
	where
	\begin{align*}
		a_{\beta}=(-1)^{|\beta|-1}\frac{2n(1-\alpha)}{\binom{m+n-1}{n-1}m}
	\end{align*}
	for $\beta=(\beta_1,\ldots,\beta_n)$ such that $|\beta|=m\geq 2$. It is easy to verify that $f_3\in \mathscr{P}_{\mathcal{H}_n^0}(\alpha)$. 
	
	Now for the point $z=(r,r,\ldots,r)$, where $0<r<1$, we find that
	\begin{align*}
		|f_3(z)|=&nr+2n(1-\alpha)\sum\limits_{m=2}^{\infty}(-1)^{m-1}\frac{r^m}{\binom{m+n-1}{n-1}m}\sum\limits_{|\beta|=m} 1\\=&
		nr+2n(1-\alpha)\sum\limits_{m=2}^{\infty}\frac{(-1)^{m-1}}{m}r^m\\=&
		\left|\sum\limits_{j=1}^nz_j\right|+2n(1-\alpha)\sum\limits_{m=2}^{\infty}(-1)^{m-1}\frac{\|z\|_{\infty}^{m}}{m}.
	\end{align*}
	Therefore the left inequality is sharp.\vspace{1.2mm}
	
	Next to show the right inequality is sharp, we consider the function $f_4(z)$ defined by
	\begin{align*}
		f_4(z)=\sum\limits_{j=1}^n z_j+\sum\limits_{m=2}^{\infty}\sum\limits_{|\beta|=m}b_{\beta}z^{\beta},
	\end{align*}
	where
	\begin{align*}
		b_{\beta}=\frac{2n(1-\alpha)}{\binom{m+n-2}{n-1}m}
	\end{align*}
	for $\beta=(\beta_1,\ldots,\beta_n)$ such that $|\beta|=m\geq 2$. It is easy to verify that $f_4\in \mathscr{P}_{\mathcal{H}_n^0}(\alpha)$.
	
	For the point $z=(r,r,\ldots,r)$, where $r<1$, we have
	\begin{align*}
		|f_4(z)|=n\|z\|_{\infty}+2n(1-\alpha)\sum\limits_{m=2}^{\infty}\frac{||z||_{\infty}^m}{m},
	\end{align*}
	which shows that the right inequality is sharp. 
\end{proof}

\section{\bf{Bohr Radius for the class $\mathscr{P}_{\mathcal{H}_n^0}(\alpha)$}}
Using Theorems \ref{Th-1.3} and \ref{Th-1.4}, we obtain the Bohr radius for the class $\mathscr{P}_{\mathcal{H}_n^0}(\alpha)$.

\begin{theo}\label{Th-1.5} Let $f=h+\ol g\in \mathscr{P}_{\mathcal{H}_n^0}(\alpha)$ be given by \emph{(\ref{Eq 1.7})} with $0<\alpha\leq 1$. Then the inequality \eqref{Eq 1.8} holds for all $z=(z_1,z_2,\ldots,z_n)\in \mathbb{P} \Delta(0;1/n)$ such that $\|z\|_{\infty}={\bf{r}}\leq {\bf{r}_f}$ and $n{\bf{r}}$ is the unique positive root of 
\begin{align*}
 n{\bf{r}}+2n(1-\alpha)\sum\limits_{m=2}^{\infty}\frac{(n{\bf{r}})^m}{m}=1+2n(1-\alpha)\sum\limits_{m=2}^{\infty}\frac{(-1)^{m-1}}{mn^m}
 \end{align*}
in $(0,1)$. The radius ${\bf{r}_f}$ is the best possible. 
\end{theo}

\begin{proof}
Let $f\in \mathscr{P}_{\mathcal{H}_n^0}(\alpha)$. 
Then by Theorem \ref{Th-1.4}, we have 
\begin{align}\label{Th5:1.1}
|f(z)|\geq \left|\sum\limits_{j=1}^nz_j\right|+2n(1-\alpha)\sum\limits_{m=2}^{\infty}
(-1)^{m-1}\frac{\|z\|_{\infty}^m}{m}
\end{align}
for all $z\in \mathbb{P}\Delta(0;1/n)$.
By taking $\liminf$ as $(|z_1|,|z_2|,\ldots,|z_n|)\to (1/n,1/n,\ldots,1/n)$ on both sides of \eqref{Th5:1.1}, we obtain
\begin{align}\label{Th5:1.2}
\liminf\limits_{(|z_1|,|z_2|,\ldots,|z_n|)\to (1/n,1/n,\ldots,1/n)}|f(z)|\geq 1+2n(1-\alpha)\sum\limits_{m=2}^{\infty}\frac{(-1)^{m-1}}{m n^m}
\end{align}

The Euclidean distance between $f(0)$ and the boundary of $f(\mathbb{P}\Delta(0;1/n))$ is given by
\begin{align}\label{Th5:1.3}
	d\left(f(0),\partial f(\mathbb{P}\Delta(0;1/n))\right)=\liminf\limits_{(|z_1|,|z_2|,\ldots,|z_n|)\to (1/n,1/n,\ldots,1/n)}|f(z)-f(0)|.
\end{align}
Since $f(0)=0$, from (\ref{Th5:1.2}) and (\ref{Th5:1.3}), we get
 \begin{align}\label{Th5:1.4}
 d\left(f(0),\partial f(\mathbb{P}\Delta(0;1/n))\right)\geq 1+2n(1-\alpha)\sum\limits_{m=2}^{\infty}\frac{(-1)^{m-1}}{m n^m}.
 \end{align}
 
 Let $H_1:[0,1)\to \mathbb{R}$ be defined by
 \begin{align*}
 H_1(n{\bf{r}})=&n{\bf{r}}+2n(1-\alpha)\sum\limits_{m=2}^{\infty}\frac{(n{\bf{r}})^m}{m}-1-2n(1-\alpha)\sum\limits_{m=2}^{\infty}\frac{(-1)^{m-1}}{mn^m}\\=&n{\bf{r}}-2n(1-\alpha)(n{\bf{r}}+\log (1-n{\bf{r}}))-1-2n(1-\alpha)\sum\limits_{m=2}^{\infty}\frac{(-1)^{m-1}}{mn^m}.
 \end{align*}
 
 Note that
 \begin{align*}
 	H_1(0)=-1-2n(1-\alpha)\sum\limits_{m=2}^{\infty}\frac{(-1)^{m-1}}{mn^m}=-1-2n(1-\alpha)(\log (1+1/n)-1/n)<0
 \end{align*}
 for all $\alpha\in [0,1)$ and. On the other hand, we see that $H_1(n{\bf{r}})\to +\infty$ as $n{\bf{r}}\to 1$ and 
 \begin{align*}
 \frac{d H_1(n{\bf{r}})}{d n{\bf{r}}}=1+2n(1-\alpha)\frac{n{\bf{r}}}{1-n{\bf{r}}}>0
 \end{align*}
 for $n{\bf{r}}\in (0,1)$. 
 
 This shows that $H_1(n{\bf{r}})$ is strictly increasing on $(0,1)$.
 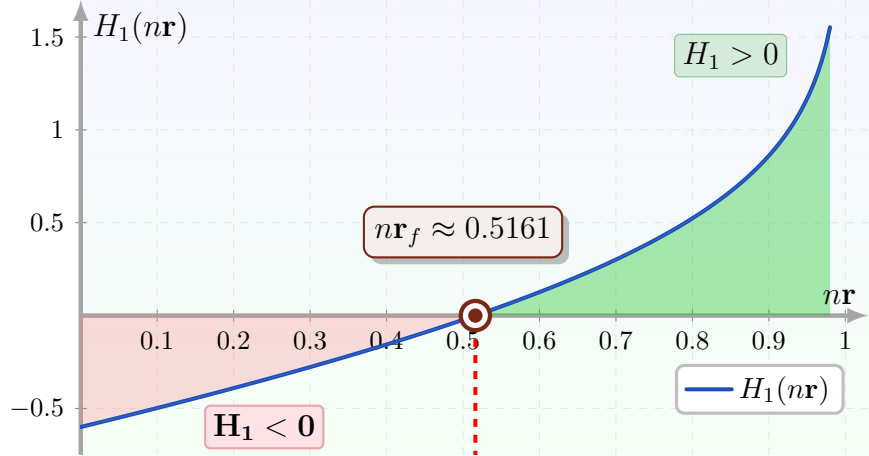
\begin{figure}[H]
 	\centering
 	\begin{tikzpicture}
 		\begin{axis}[
 			width=12cm, height=7.6cm,                 
 			axis lines=middle,
 			xlabel={$n\mathbf{r}$},
 			ylabel={$H_1(n\mathbf{r})$},
 			xmin=0, xmax=1.03,
 			ymin=-.75, ymax=1.70,
 			samples=180,
 			domain=0:0.98,
 			very thick,
 			axis line style={-latex, gray!70, line width=1.7pt},
 		tick label style={font=\footnotesize},
 			label style={font=\bfseries\normalsize},
 			grid=major,
 			grid style={gray!20, dashed},
 			minor grid style={gray!10},
 			axis background/.style={
 				shading=axis,
 				top color=blue!4,
 				bottom color=green!4
 			},
 			legend style={
 				draw=gray!50,
 				fill=white,
 				rounded corners=3pt,
 				%drop shadow,
 				font=\small\bfseries,
 				at={(0.97,0.20)},
 				anchor=north east,
 				row sep=2pt
 			}
 			]

 			\addplot[
 			curveblue,
 			opacity=0.15,
 			line width=1.5pt,
 			line cap=round,
 			domain=0:0.98,
 			samples=180,
 			forget plot]
 			{x-0.4*(x+ln(1-x))-0.6};

 			\addplot[name path=zeroline, draw=none, domain=0:0.98, forget plot] {0};

 			\addplot[
 			name path=Hcurve,
 			curveblue!85!black,
 			line width=1.5pt,
 			line cap=round,
 			line join=round,
 			domain=0:0.98,
 			samples=180]
 			{x-0.4*(x+ln(1-x))-0.6};
 			\addlegendentry{$H_1(n\mathbf r)$}

 	\addplot[
 	white,
 	opacity=0.15,
 	line width=0.35pt,
 	domain=0:0.98,
 	samples=180,
 	forget plot]
 	{x-0.4*(x+ln(1-x))-0.6};

 			\addplot[negred!35!red, opacity=0.15, forget plot]
 			fill between[of=Hcurve and zeroline, soft clip={domain=0:0.5161}];
 			\addplot[negred, opacity=0.12, forget plot]
 			fill between[of=Hcurve and zeroline, soft clip={domain=0:0.5161}];

 			\addplot[posgreen!35!green, opacity=0.28, forget plot]
 			fill between[of=Hcurve and zeroline, soft clip={domain=0.5161:0.98}];
 			\addplot[posgreen, opacity=0.12, forget plot]
 			fill between[of=Hcurve and zeroline, soft clip={domain=0.5161:0.98}];

 			\addplot[dashed, red, thick, line width=1.5pt, forget plot]
 			coordinates {(0.5161,-1.2) (0.5161,0)};

 			\addplot[only marks, mark=*, mark size=5.5, forget plot,
 			mark options={fill=white, draw=rootdark, line width=1.6pt}]
 			coordinates {(0.5161,0)};
 			\addplot[only marks, mark=*, mark size=3.2, forget plot,
 			mark options={fill=rootdark, draw=white, line width=1pt}]
 			coordinates {(0.5161,0)};

 			\node[fill=rootdark!8, draw=rootdark, rounded corners=4pt, thick,
 			drop shadow, font=\bfseries, inner sep=4pt]
 			at (axis cs:0.50,0.45) {$n\mathbf{r}_f\approx 0.5161$};

 			\node[
 			text=black,
 			fill=negred!40,
 			draw=negred!90!black,
 			rounded corners=2pt,
 			thick,
 			font=\bfseries,
 			inner sep=3pt
 			]
 			at (axis cs:0.24,-0.60) {$\mathbf{H_1<0}$};
 			\node[ fill=posgreen!20, rounded corners=2pt, draw=posgreen!60,
 			font=\itshape\bfseries, inner sep=3pt]
 			at (axis cs:0.85,1.4) {$H_1>0$};
 			
 		\end{axis}
 	\end{tikzpicture}
 	\caption{Graph of $H_1(n\mathbf r)$ illustrating its unique zero at
 		$n\mathbf r_f\approx0.5161$, which determines the sharp Bohr radius
 		$\mathbf r_f$. The shaded regions indicate $H_1<0$ (red) and $H_1>0$
 		(green).}
 	\label{fig:H1-unique-zero}
 \end{figure}
 Since $H_1(0)<0$ and $H_1(n{\bf{r}})\to +\infty$ as $n{\bf{r}}\to 1$, the monotonicity of $H_1(n{\bf{r}})$ implies that $H_1(n{\bf{r}})$ has exactly one zero in $(0,1)$. Let $n{\bf{r}}_f$ be the unique root of $H_1(n{\bf{r}})$ in $(0,1)$. Then $H_1(n{\bf{r}}_f)=0$ implies that 
 \begin{align}\label{Th5:1.5}
 n{\bf{r}}_f+2n(1-\alpha)\sum\limits_{m=2}^{\infty}\frac{(n{\bf{r}}_f)^m}{m}=1+2n(1-\alpha)\sum\limits_{m=2}^{\infty}\frac{(-1)^{m-1}}{mn^m}.
 \end{align}
   \begin{figure}[H]
 	\centering
 	\begin{tikzpicture}
 		\begin{axis}[
 			width=13cm,
 			height=9.5cm,
 			xlabel={$\alpha$},
 			ylabel={$\mathbf{r}_{_{f}}$},
 			xmin=-0.08, xmax=1.08,
 			ymin=0, ymax=0.88,
 			grid=major,
 			grid style={gray!30, line width=0.3pt},
 			thick,
 			minor tick num=1,
 			minor grid style={gray!10},
 			axis line style={
 				-{Latex[length=3mm,width=2mm]},
 				black,
 				line width=0.9pt
 			},
 			tick label style={font=\small},
 			label style={font=\small\bfseries},
 			legend style={draw=gray!40, fill=white, font=\small,
 				at={(0.02,0.98)}, anchor=north west, rounded corners}
 			]

 			\addplot[mark=*, mark size=1.5pt, c1, thick]
 			coordinates {(0,0.2852) (0.25,0.3866) (0.5,0.5000) (0.75,0.6483) (0.9,0.7868)};
 			\addlegendentry{$n=1$}

 			\addplot[mark=square*, mark size=1.5pt, c2, thick]
 			coordinates {(0,0.1667) (0.25,0.1973) (0.5,0.2370) (0.75,0.2978) (0.9,0.3651)};
 			\addlegendentry{$n=2$}

 			\addplot[mark=triangle*, mark size=1.5pt, c3, thick]
 			coordinates {(0,0.1077) (0.25,0.1243) (0.5,0.1470) (0.75,0.1844) (0.9,0.2291)};
 			\addlegendentry{$n=3$}

 			\addplot[mark=diamond*, mark size=1.5pt, c4, thick]
 			coordinates {(0,0.0768) (0.25,0.0877) (0.5,0.1033) (0.75,0.1299) (0.9,0.1634)};
 			\addlegendentry{$n=4$}

 			\node[c1, font=\tiny\bfseries] at (axis cs:0,0.2852)   [above,      yshift=3pt] {0.2852};
 			\node[c1, font=\tiny\bfseries] at (axis cs:0.25,0.3866)[above,      yshift=3pt] {0.3866};
 			\node[c1, font=\tiny\bfseries] at (axis cs:0.5,0.5000) [above,      yshift=3pt] {0.5000};
 			\node[c1, font=\tiny\bfseries] at (axis cs:0.75,0.6483)[above,      yshift=3pt] {0.6483};
 			\node[c1, font=\tiny\bfseries] at (axis cs:0.9,0.7868) [above,      yshift=3pt] {0.7868};

 			\node[c2, font=\tiny\bfseries] at (axis cs:0,0.1667)   [above,      yshift=3pt] {0.1667};
 			\node[c2, font=\tiny\bfseries] at (axis cs:0.25,0.1973)[above,      yshift=3pt] {0.1973};
 			\node[c2, font=\tiny\bfseries] at (axis cs:0.5,0.2370) [above,      yshift=3pt] {0.2370};
 			\node[c2, font=\tiny\bfseries] at (axis cs:0.75,0.2978)[above,      yshift=3pt] {0.2978};
 			\node[c2, font=\tiny\bfseries] at (axis cs:0.9,0.3651) [above,      yshift=3pt] {0.3651};

 			\node[c3, font=\tiny\bfseries, fill=white, draw=c3, rounded corners,
 			inner sep=1pt, line width=0.4pt]
 			at (axis cs:0,0.1077)    [above, yshift=4pt]  {0.1077};
 			\node[c3, font=\tiny\bfseries, fill=white, draw=c3, rounded corners,
 			inner sep=1pt, line width=0.4pt]
 			at (axis cs:0.25,0.1243) [above, yshift=6pt]  {0.1243};
 			\node[c3, font=\tiny\bfseries, fill=white, draw=c3, rounded corners,
 			inner sep=1pt, line width=0.4pt]
 			at (axis cs:0.5,0.1470)  [above, yshift=6pt]  {0.1470};
 			\node[c3, font=\tiny\bfseries, fill=white, draw=c3, rounded corners,
 			inner sep=1pt, line width=0.4pt]
 			at (axis cs:0.75,0.1844) [above, yshift=6pt]  {0.1844};
 			\node[c3, font=\tiny\bfseries, fill=white, draw=c3, rounded corners,
 			inner sep=1pt, line width=0.4pt]
 			at (axis cs:0.9,0.2291)  [above, yshift=6pt]  {0.2291};

 			\node[c4, font=\tiny\bfseries, fill=white, draw=c4, rounded corners,
 			inner sep=1pt, line width=0.4pt]
 			at (axis cs:0,0.0768)    [below, yshift=-6pt] {0.0768};
 			\node[c4, font=\tiny\bfseries, fill=white, draw=c4, rounded corners,
 			inner sep=1pt, line width=0.4pt]
 			at (axis cs:0.25,0.0877) [below, yshift=-6pt] {0.0877};
 			\node[c4, font=\tiny\bfseries, fill=white, draw=c4, rounded corners,
 			inner sep=1pt, line width=0.4pt]
 			at (axis cs:0.5,0.1033)  [below, yshift=-6pt] {0.1033};
 			\node[c4, font=\tiny\bfseries, fill=white, draw=c4, rounded corners,
 			inner sep=1pt, line width=0.4pt]
 			at (axis cs:0.75,0.1299) [below, yshift=-6pt] {0.1299};
 			\node[c4, font=\tiny\bfseries, fill=white, draw=c4, rounded corners,
 			inner sep=1pt, line width=0.4pt]
 			at (axis cs:0.9,0.1634)  [below, yshift=-6pt] {0.1634};
 			
 		\end{axis}
 	\end{tikzpicture}
 	\caption{The sharp Bohr radius $\mathbf{r}_f$ of Theorem~\ref{Th-1.5}, obtained as the unique solution of \eqref{Th5:1.5}, plotted against $\alpha$ for $n=1,2,3,4$. The radius increases with $\alpha$ and decreases with $n$.}
 	\label{fig:rf-vs-alpha}
 \end{figure}
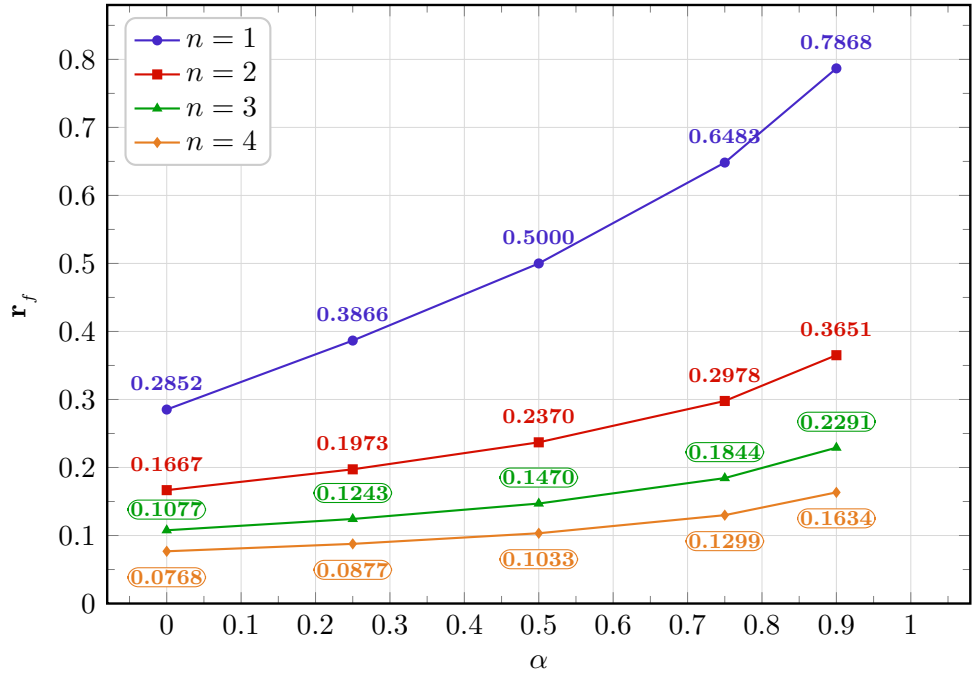
 
 For $0\leq n{\bf{r}}\leq n{\bf{r}}_f$, we deduce from (\ref{Th5:1.5}) that 
 \begin{align}\label{Th5:1.6}
 n{\bf{r}}+2n(1-\alpha)\sum\limits_{m=2}^{\infty}\frac{(n{\bf{r}})^m}{m}\leq&\; n{\bf{r}}_f+2n(1-\alpha)\sum\limits_{m=2}^{\infty}\frac{(n{\bf{r}}_f)^m}{m}\\=&1+2n(1-\alpha)\sum\limits_{m=2}^{\infty}\frac{(-1)^{m-1}}{mn^m}.\nonumber
 \end{align}
 
 Let $z\in \mathbb{P}\Delta(0;1/n)$ such that $r=(|z_1|,|z_2|,\ldots,|z_n|)=(r_1,r_2,\ldots,r_n)$ and ${\bf{r}}=\|z\|_{\infty}$. 
Using Theorem \ref{Th-1.3} and the inequalities \eqref{Th5:1.4} and \eqref{Th5:1.6}, for $0\le n\|z\|_{\infty}=n{\bf{r}}\leq n{\bf{r}}_f$, we deduce that 
 \begin{align*}
 \left|\sum\limits_{j=1}^n z_j\right|+\sum\limits_{m=2}^{\infty}\left(\left|\sum\limits_{|\beta|=m}a_{\beta}\right|+\left|\sum\limits_{|\beta|=m}b_{\beta}\right|\right)r^{\beta}\leq&
 n\|z\|_{\infty}+2n(1-\alpha)\sum\limits_{m=2}^{\infty}\binom{m+n-1}{n-1}\frac{\|z\|_{\infty}^m}{m}\\\leq&
 n{\bf{r}}+2n(1-\alpha)\sum\limits_{m=2}^{\infty}\frac{(n{\bf{r}})^m}{m}\\\leq&
 1+2n(1-\alpha)\sum\limits_{m=2}^{\infty}\frac{(-1)^{m-1}}{mn^m}\\=& d\left(f(0),\partial f(\mathbb{P}\Delta(0;1/n))\right).
 \end{align*}
 
\begin{table}[H]
	\centering
	\renewcommand{\arraystretch}{1.35}
	\setlength{\tabcolsep}{8pt}
	\begin{tabular}{|p{4.2cm}|p{8.3cm}|}
		\hline
		\textbf{Step of the proof} &
		\textbf{Mathematical conclusion}\\
		\hline
		Coefficient estimate
		&
		Theorem~\ref{Th-1.3} provides an upper bound for the coefficient sum
		$\displaystyle
		\sum_{|\beta|=m}\bigl(|a_\beta|+|b_\beta|\bigr).
		$
		\\
		\hline
		Growth estimate
		&
		Theorem~\ref{Th-1.4} yields the lower bound for
		$|f(z)|$, leading to the estimate
		(\ref{Th5:1.4}) for
		$d\!\left(f(0),\partial f(\mathbb P\Delta(0;1/n))\right)$.
		\\
		\hline
		Auxiliary function
		&
		The function
		$H_1(n{\bf r})$
		is introduced and proved to be strictly increasing on
		$(0,1)$, ensuring the existence of a unique positive root.
		\\
		\hline
		Bohr radius
		&
		The unique root
		${\bf r}_f$
		is characterized by
		(\ref{Th5:1.5}).
		\\
		\hline
		Bohr inequality
		&
		Equation
		(\ref{Th5:1.6})
		establishes the desired Bohr inequality for
		$0\le n{\bf r}\le n{\bf r}_f$.
		\\
		\hline
		Sharpness
		&
		The extremal mapping
		$f_5$
		shows that the radius
		${\bf r}_f$
		cannot be improved.
		\\
		\hline
	\end{tabular}
	\caption{Logical progression of the proof of Theorem~\ref{Th-1.5}.}
	\label{tab:proof-outline-th15}
\end{table}
 To show that the constant ${\bf{r}}_f$ is the best possible, we consider the function $f_5(z)$ defined by 
 \begin{align*}
 	f_5(z)=\sum\limits_{j=1}^n z_j+\sum\limits_{m=2}^{\infty}\sum\limits_{|\beta|=m}a_{\beta}z^{\beta},\;\;
 	a_{\beta}=\frac{2n(1-\alpha)n^{\beta}}{\binom{m+n-1}{n-1}m}
 \end{align*}
 for $\beta=(\beta_1,\ldots,\beta_n)$ such that $|\beta|=m\geq 2$.

 Let $z=(r,r,\ldots,r)$, where $nr<1$. Now using \eqref{Th5:1.5}, we deduce that

 \begin{align*}
 	&\left|\sum\limits_{j=1}^n z_j\right|+\sum\limits_{m=2}^{\infty}\left(\left|\sum\limits_{|\beta|=m}a_{\beta}\right|+\left|\sum\limits_{|\beta|=m}b_{\beta}\right|\right)r^{\beta} =
 	n\|z\|_{\infty}+2n(1-\alpha)\sum\limits_{m=2}^{\infty}\frac{(n\|z\|_{\infty})^m}{m}\\=&
 	n{\bf{r}}_f+2n(1-\alpha)\sum\limits_{m=2}^{\infty}\frac{(n{\bf{r}}_f)^m}{m}=
 	1+2n(1-\alpha)\sum\limits_{m=2}^{\infty}\frac{(-1)^{m-1}}{mn^m}\\=& d\left(f(0),\partial f(\mathbb{P}\Delta(0;1/n))\right)
 \end{align*}
 and hence the radius ${\bf{r}}_f$ is the best possible.
\end{proof}

\section{{\bf Partial sums of functions in $\mathscr{P}_{\mathcal{H}_n^0}(\alpha)$}}\label{Sec-4}
Let 
\begin{align*}
f(z)=\sum\limits_{j=1}^n z_j+\sum\limits_{m=2}^{\infty}\sum\limits_{|\beta|=m}a_{\beta}z^{\beta}
\end{align*}
for $z\in \mathbb{P}\Delta(0_n;1_n)$. Then the $p$-th partial sum (or section) of $f(z)$ is defined by
\begin{align*}
S_p(f(z))=\sum\limits_{j=1}^n z_j+\sum\limits_{m=2}^{p}\sum\limits_{|\beta|=m}a_{\beta}z^{\beta}
\end{align*}
for $z\in \mathbb{P}\Delta(0_n;1_n)$ and $p\geq 2$. Analogously in the pluriharmonic case, the $p, q$-th partial sum (or section) of a harmonic function $f=h + \ol g$ given by (\ref{Eq 1.7}) is defined by
\begin{align*}
S_{p,q}(f)=S_p(h)+\ol{S_q(g)},
\end{align*}
where
\begin{align*}
S_p(h(z))=\sum\limits_{j=1}^n z_j+\sum\limits_{m=2}^{p}\sum\limits_{|\beta|=m}a_{\beta}z^{\beta}\;\;\text{and}\;\;S_q(g(z))=\sum\limits_{m=2}^{q}\sum\limits_{|\beta|=m}b_{\beta}z^{\beta}
\end{align*}
for $z\in \mathbb{P}\Delta(0_n;1_n)$ and $p,q\geq 2$.\vspace{1.2mm}

In 1928, Szeg\"{o} \cite{Szego-1928} proved the remarkable result that every section (or partial sum)
 $S_p(f)$ of a function $f\in\mathcal{A}$ is univalent in the disk $|z|<1/4$. The radius $\frac{1}{4}$ is sharp (best possible), as can be seen by examining the second partial sum of the Koebe function: $k(z)= z/(1-z)^2$.\vspace{1.2mm}

The study of sections (or partial sums) of harmonic mappings is primarily motivated by the utility of approximating real-valued harmonic functions with harmonic polynomials (see \cite{Walsh-1929}). Such approximations offer several key advantages. For instance, by the maximum principle, a harmonic function attains its extreme values on the boundary of its domain.

Since any planar harmonic mapping $f = h + \overline{g}$ defined on the unit disk $\mathbb{D}$ admits a canonical power series representation, its sections serve as natural approximations by complex-valued harmonic polynomials. Consequently, approximating univalent harmonic mappings with univalent harmonic polynomials can yield new theoretical insights and practical applications, particularly in modeling physical phenomena in fluid dynamics and flow theory. \vspace{1.2mm}

\begin{table}[H]
	\centering
	\renewcommand{\arraystretch}{1.25}
	\begin{tabular}{|c|c|}
		\hline
		Section & Definition\\
		\hline
		$S_p(h)$
		&
		$\displaystyle
		\sum_{k=1}^{n}z_k+
		\sum_{m=2}^{p}
		\sum_{|\beta|=m}
		a_\beta z^\beta$
		\\
		\hline
		$S_q(g)$
		&
		$\displaystyle
		\sum_{m=2}^{q}
		\sum_{|\beta|=m}
		b_\beta z^\beta$
		\\
		\hline
		$S_{p,q}(f)$
		&
		$S_p(h)+\overline{S_q(g)}$
		\\
		\hline
	\end{tabular}
	\caption{Partial sums used throughout this section.}
\end{table}

In $2013$, Li and Ponnusamy \cite{Li-Ponnusamy-2013} discussed the sections of functions in the class $\mathscr{P}^{0}_{\mathcal{H}}$
and in 2015, Ponnusamy et al. \cite{Li-Ponnusamy-2015} also investigated properties of the sections of stable harmonic convex functions (see also \cite{Ponnusamy-Kaliraj-Starkov-2017}). For several other interesting results concerning sections of analytic functions, we refer the reader to \cite{Obradovic-Ponnusamy-2013, Obradovic-Ponnusamy-2014, Ponnusamy-Sahoo-Yanagihara-2014,Silverman-1988, Singh-1970}.\vspace{1.2mm}

Regarding sections of harmonic mappings, Li and Ponnusamy
\cite{Li-Ponnusamy-2013a} obtained the following result.

\begin{theoF}\emph{\cite[Theorem 4]{Li-Ponnusamy-2013a}}
 For every $q\ge2$, $S_{1,q}(f)\in \mathscr{P}_{\mathcal{H}}^0$ for $|z|< \frac{1}{2}$.
\end{theoF}
 
In this section, we introduce the sections of pluriharmonic mappings in the unit polydisk $\mathbb{P}\Delta(0;1_n)$ and obtain the following theorem.
\begin{table}[H]
	\centering
	\renewcommand{\arraystretch}{1.25}
	\caption{Comparison of partial-sum radii.}
	\label{tab:sections-comparison}
	\begin{tabular}{|l|c|c|}
		\hline
		Result & Radius & Setting \\
		\hline
		Szeg\H{o} (1928) & $1/4$ & $f\in\mathcal{A}$, $n=1$, analytic \\
		\hline
		Li--Ponnusamy (2013), Theorem F & $1/2$ & $f\in \mathscr{P}_{\mathcal{H}}^0$, $n=1$ \\
		\hline
		\textbf{Theorem \ref{Th7-1} (present paper)} & $1/(2n^2)$ & $f\in \mathscr{P}_{\mathcal{H}_n^0}(\alpha)$, $n\geq 1$ \\
		\hline
	\end{tabular}
\end{table}

\begin{theo}\label{Th7-1} Let $f=h+\ol g\in \mathscr{P}_{\mathcal{H}_n^0}(\alpha)$ with $0\leq \alpha<1$ and be given by \emph{(\ref{Eq 1.7})}. Then for each $q \geq 2$, $S_{1,q}(f)\in \mathscr{P}_{\mathcal{H}_n^0}(\alpha)$ for $\|z\|_{\infty}< \frac{1}{2n^2}$.	
\end{theo}
\begin{proof}
Let $f=h+\ol g\in \mathscr{P}_{\mathcal{H}_n^0}(\alpha)$, where $h$ and $g$ are given by (\ref{Eq 1.7}). Note that
\begin{align*}
	S_{1,q}(f(z))=S_1(h(z))+\ol{S_q(g(z))}=\sum\limits_{j=1}^n z_j+\ol{\sum\limits_{m=2}^q\sum_{|\beta|=m}b_{\beta}z^{\beta}}.
\end{align*}

Also we see that
\begin{align*}
	\frac{\partial S_1(h(z))}{\partial z_j}=1
\end{align*}
and so 
\begin{align*}
\sum\limits_{j=1}^n\Re\left(\frac{\partial S_1(h(z))}{\partial z_j}-\alpha\right)=n(1-\alpha).
\end{align*}

On the other hand, from (\ref{Th2:1.5a}) and (\ref{Th2:1.8a}), we deduce that
\begin{align}\label{Th7:1.1}
&\sum\limits_{k=1}^n\left|\frac{\partial S_q(g(z))}{\partial z_k}\right|\\\leq&
\sum\limits_{k=1}^n\sum\limits_{m=2}^{q}\sum\limits_{|\beta|=m}\beta_{k} \left|b_{\beta}z_1^{\beta_1}\ldots z_{j-1}^{\beta_{j-1}}z_{j}^{\beta_j}z_{j+1}^{\beta_{j+1}}\ldots z_{k-1}^{\beta_{k-1}}z_{k}^{\beta_k-1}z_{k+1}^{\beta_{k+1}}\ldots z_n^{\beta_n}\right|\nonumber
\\\leq &
\sum\limits_{k=1}^n\sum\limits_{m=2}^{q}\sum\limits_{|\beta|=m}\beta_{k} \left|b_{\beta}\right|\|z\|_{\infty}^{m-1}
\nonumber\\=&
\sum\limits_{m=2}^{q}\sum\limits_{|\beta|=m}m|b_{\beta}|\|z\|_{\infty}^{m-1}\nonumber.
\end{align}
Using Theorem \ref{Th-1.2} to \eqref{Th7:1.1}, we obtain
\begin{align*}
 \sum\limits_{k=1}^n\left|\frac{\partial S_q(g(z))}{\partial z_k}\right|\leq&n(1-\alpha) \sum\limits_{m=2}^q \binom{m+n-1}{n-1}\|z\|_{\infty}^{m-1}\\\leq&n(1-\alpha)
 \sum\limits_{m=2}^q n^{m}\|z\|_{\infty}^{m-1}\\<&n(1-\alpha)
 \sum\limits_{m=2}^q  \frac{n^{m}}{2^{m-1}n^{2(m-1)}}\\\leq&n(1-\alpha)
 \sum\limits_{m=2}^q \frac{1}{2^{m-1}}\\<& n(1-\alpha)=\sum\limits_{j=1}^n\Re\left(\frac{\partial S_1(h(z))}{\partial z_j}-\alpha\right).
 \end{align*}
 
 Therefore $S_{1,q}(f)\in \mathscr{P}_{\mathcal{H}_n^0}(\alpha)$.
\end{proof}

\begin{cor}\label{cor-1} Let $f=h+\ol g\in \mathscr{P}_{\mathcal{H}_n^0}$ and be given by \emph{(\ref{Eq 1.7})}. Then for each $q \geq 2$, $S_{1,q}(f)\in \mathscr{P}_{\mathcal{H}_n^0}$ for $\|z\|_{\infty}< \frac{1}{2n^2}$.	
\end{cor}

\vspace{5mm}

\noindent\textbf{Conflict of interest:} The authors declare that there is no conflict  of interest regarding the publication of this paper.\vspace{1.2mm}

\noindent {\bf Funding:} Not Applicable.\vspace{1.2mm}

\noindent\textbf{Data availability statement:}  Data sharing not applicable to this article as no datasets were generated or analysed during the current study.\vspace{1.2mm}

\noindent {\bf Authors' contributions:} All the authors have equal contributions in preparation of the manuscript.

\end{document}